\documentclass[11pt]{amsart}
\usepackage{amsfonts, amsmath, amssymb, amsthm, color, bookmark}

\usepackage{hyperref}
\hypersetup{
  pdfborder={0 0 0},
  colorlinks   = true, 
  urlcolor     = blue, 
  linkcolor    = blue, 
  citecolor   = red 
}

\newtheorem{thm}{Theorem}[section]
\newtheorem{prop}[thm]{Proposition}
\newtheorem{lemma}[thm]{Lemma}
\newtheorem{cor}[thm]{Corollary}

\theoremstyle{definition}
\newtheorem{defn}[thm]{Definition}
\newtheorem{question}[thm]{Question}

\theoremstyle{remark}
\newtheorem{ex}[thm]{Example}
\newtheorem{rem}[thm]{Remark}

\newcommand{\C}{{\rm C}}
\newcommand{\E}{{\rm E}}

\newcommand{\e}{\varepsilon}

\DeclareMathOperator{\Fix}{Fix}

\newcommand{\cC}{\mathcal{C}}
\newcommand{\proc}{{pro-$\mathcal C$ topology }}
\newcommand{\FF}{F_1\times F_2}

\newcommand{\N }{\mathbb{N}}
\newcommand{\Z }{\mathbb{Z}}

\def\coloneq{\mathrel{\mathop\mathchar"303A}\mkern-1.2mu=}

\begin{document}

\title{On conjugacy separability of fibre products}

\author{Ashot Minasyan}
\address{Mathematical Sciences,
University of Southampton, Highfield, Southampton, SO17 1BJ, United
Kingdom.}
\email{aminasyan@gmail.com}
%

\begin{abstract} In this paper we study conjugacy separability of subdirect products of two free (or hyperbolic) groups.
We establish necessary and sufficient criteria and apply them to fibre products to produce a finitely presented group $G_1$ in which all finite index
subgroups are conjugacy separable,  but which has an index $2$ overgroup that is not conjugacy separable. Conversely, we construct a finitely presented group
$G_2$ which has a non-conjugacy separable subgroup of index $2$ such that every finite index normal overgroup of $G_2$ is conjugacy separable. The normality of the overgroup
is essential in the last example, as such a group $G_2$ will always posses an index $3$ overgroup that is not conjugacy separable.

Finally, we characterize $p$-conjugacy separable subdirect products of two free groups, where $p$ is a prime.
We show that fibre products provide a natural correspondence between residually finite $p$-groups and $p$-conjugacy separable subdirect products of two
non-abelian free groups. As a consequence, we deduce that the open question about the existence of an infinite finitely presented residually
finite $p$-group is equivalent to the question about
the existence of a finitely generated $p$-conjugacy separable full subdirect product of
infinite index in the direct product of two free groups.
\end{abstract}


\maketitle

\section{Introduction}
Let $\mathcal{C}$ be a class of groups. A group $G$ is said to be $\cC$-conjugacy separable if one can distinguish its conjugacy classes by looking at the
quotients of $G$ in $\cC$. More precisely, $G$ is \emph{$\cC$-conjugacy separable} for any pair of non-conjugate elements
$x,y \in G$ there must exist a group $M \in \cC$ and a homomorphism
$\varphi:G \to M$ such that $\varphi(x)$ and $\varphi(y)$ are not conjugate in $M$. In the case when $\cC$ is the class of all \emph{finite groups}, we omit the ``$\cC$-''
and simply write that $G$ is \emph{conjugacy separable}; and if $\cC=\cC_p$ is the class of all \emph{finite $p$-groups}, for some prime $p$, we will say that
$G$ is \emph{$p$-conjugacy separable}.

Conjugacy separability is a basic and classical residual property of a group that is closely related to the solvability of the conjugacy problem.
It is a natural strengthening of residual finiteness, which corresponds to the solvability of the word problem.

Many groups are known to be residually finite, including all finitely generated linear groups (Mal'cev \cite{Malcev-lin_gps}).
Conjugacy separability, though, is a more delicate property. It is often difficult to check whether a given residually finite group is conjugacy separable.
Classical examples of conjugacy separable groups include virtually polycyclic groups (Remeslennikov \cite{Rem-polyc}, Formanek \cite{Formanek}) and
virtually free groups (Dyer \cite{Dyer}). The development of geometric methods in Group Theory prompted a lot of recent progress in this field, and the following classes
of groups were found to be conjugacy separable:

\begin{itemize}
  \item virtually surface groups and Seifert fibered $3$-manifold groups (Martino \cite{Martino});
  \item limit groups and, more generally, finitely presented residually free groups (Chagas and Zalesskii \cite{Chag-Zal-lim,Chag-Zal-fp_res_free});
  \item right angled Artin groups (Minasyan \cite{M-RAAG});
  \item most even Coxeter groups (Caprace and Minasyan \cite{Cap-Min});
  \item $1$-relator groups with torsion (Minasyan and Zalesskii \cite{M-Z-1_rel});
  \item compact orientable $3$-manifold groups (Hamilton, Wilton and Zalesskii \cite{H-W-Z});
  \item virtually compact special (Gromov) hyperbolic groups (Minasyan and Zalesskii \cite{M-Z-vcs}).
\end{itemize}

The most basic example of a residually finite group which is not conjugacy separable is $\mathrm{SL}(n,\Z)$ (Remeslennikov \cite{Rem}, Stebe \cite{Stebe-SL3}), for $n \ge 3$.
The proof of this fact uses the positive solution of the congruence subgroup problem for $\mathrm{SL}(n,\Z)$, when $n \ge 3$,
which was established by Bass, Lazard and Serre \cite{B-L-S}. An example of a non-conjugacy separable finitely presented torsion-free
metabelian group was given by Wehrfritz \cite{Wehr}.

Presently it is unknown whether all (Gromov) hyperbolic groups are conjugacy separable. In \cite{Wise-QJM} Wise noted that this question is closely linked to
the well-known open problem  asking if there exists a non-residually finite hyperbolic group.

Conjugacy separability has two main applications:
\begin{itemize}
  \item a finitely presented conjugacy separable group has solvable conjugacy problem (Mos\-towski \cite{Mostowski});
  \item a finitely generated conjugacy separable group $G$,  all of whose pointwise inner automorphisms are inner,  has a residually finite outer automorphism group $Out(G)$
  (Grossman \cite{Grossman}).
\end{itemize}

Much like the conjugacy problem, conjugacy separability may not behave well under passing to finite index subgroups and overgroups.
Finitely presented conjugacy separable groups with non-conjugacy separable subgroups of any finite index were constructed by Martino and the author in \cite{M-M}.
A finitely generated conjugacy separable group possessing an overgroup of index $2$ with unsolvable conjugacy problem was found by Goryaga in \cite{Gor}.
One can show that the overgroup from Goryaga's example is not conjugacy separable, even though this group is not finitely presented
(and thus Mostowski's result \cite{Mostowski} does not apply to it directly).

In this paper we investigate conjugacy separability of subdirect products of `non-positively curved' (e.g., free, hyperbolic or acylindrically hyperbolic) groups.
Recall that a subgroup $G \leqslant F_1\times F_2$, of a direct product of two groups $F_1,F_2$, is called a \emph{subdirect product}, if for each $i\in \{1,2\}$  the image
of $G$ under the natural projection $\rho_i:F_1\times F_2 \to F_i$ is all of $F_i$: $\rho_i(G)=F_i$. If, in addition, $G \cap F_i \neq \{1\}$ for each $i=1,2$, then
$G$ is said to be a \emph{full subdirect product} of $F_1$ and $F_2$. Note that if  $G$ is subdirect in $\FF$ then $G \cap F_i \lhd F_i$, $i=1,2$
(see Lemma~\ref{lem:norm_in_G->norm_in_F}).

A standard way for constructing subdirect products is to use \emph{fibre products} (see Subsection \ref{subsec:constr-subdir}).
It provides a streamlined and powerful method for producing groups with exotic behaviour.
The original idea belongs to Miha\u{\i}lova \cite{Mih}, who applied the fibre product construction to a finitely presented group
with unsolvable word problem to  give an example of a finitely generated subgroup $G\leqslant F \times F$, where $F$ is the free group of rank $2$, such that the membership problem for $G$ in $F\times F$ is undecidable. In \cite{Miller-book} Miller showed that the same group $G$ also has unsolvable conjugacy problem. Our first result shows, in the same spirit, that if one starts with a group which is not residually finite then the corresponding fibre product of free groups will not be conjugacy separable.

\begin{thm}\label{thm:non-cs_for_subdir_of_hyp}
Suppose that $\cC$ is a pseudovariety of groups and $F_i$ is either a non-abelian free group (of arbitrary rank)
or a non-elementary hyperbolic group without non-trivial finite normal subgroups, $i=1,2$. If $G \leqslant \FF$
is a full subdirect product such that $F_1/N_1$ is not residually-$\cC$ (where $N_1 \coloneq G \cap F_1$) then $G$ is not $\cC$-conjugacy separable.
\end{thm}

A special case of Theorem \ref{thm:non-cs_for_subdir_of_hyp}, when $N_1$ is finitely generated and $\cC$ is the class of all finite groups,
was proved by Martino and the author in \cite[Prop. 7.6]{M-M}.
The proof of Theorem~\ref{thm:non-cs_for_subdir_of_hyp} is essentially done by a direct computation and does not require the groups $F_1$, $F_2$ or the
subgroup $N_1\lhd F_1$ to be finitely generated. This makes a difference, as Theorem \ref{thm:non-cs_for_subdir_of_hyp} can be applied to fibre products
of free groups.

The first application of Theorem \ref{thm:non-cs_for_subdir_of_hyp} is in Example~\ref{ex:B-G_gp}. It starts with Baumslag's
non-residually finite $1$-relator group \cite{B-gp}, and shows that the corresponding symmetric fibre product  $G \leqslant F \times F$, where $F$ is the free group
of rank $2$, is not conjugacy separable. In fact we explicitly exhibit a pair  of non-conjugate elements of $G$ that are conjugate in every finite quotient.
Thus we get a $3$-generated residually free group that is not conjugacy separable. This can be contrasted with the result of Chagas and Zalesskii \cite{Chag-Zal-fp_res_free}
mentioned above. Moreover, note that $3$ is optimal, as any $2$-generated residually free group is either free or abelian, and so it is conjugacy separable.

A group $G$ is said to be $\cC$-\emph{hereditarily conjugacy separable} if for every subgroup $H \leqslant G$, open in the pro-$\cC$ topology on $G$, $H$
is $\cC$-conjugacy separable.
In Subsection \ref{subsec:crit_for_cs} we give a sufficient criterion for $\cC$-conjugacy separability of subdirect products
(see Proposition~\ref{prop:crit_of_CS_for_subdirect}), and in Subsection~\ref{subsect:C-cs_for_fin_ind} we combine this criterion with
Theorem \ref{thm:non-cs_for_subdir_of_hyp} to obtain the following complete characterization of $\cC$-conjugacy separability for subdirect products of finite index.

\begin{cor}\label{cor:crit_for_c-cs_if_fin_ind} Let $\cC$ be a non-trivial extension-closed pseudovariety of finite groups.
Let $F_i$ either be a non-abelian free group or a non-elementary $\cC$-hereditarily conjugacy separable hyperbolic group
without non-trivial finite normal subgroups, $i=1,2$.
If $G \leqslant \FF$ is a subdirect product of finite index in $\FF$ then the following statements are equivalent:
\begin{itemize}
  \item[(1)] $G$ is $\cC$-conjugacy separable;
  \item[(2)] $F_1/N_1 \in \cC$, where $N_1 \coloneq G \cap F_1$;
  \item[(3)] $G$ is open in the \proc on $\FF$.
\end{itemize}
\end{cor}

If $\cC=\cC_p$ is the class of all finite $p$-groups, for some prime $p$, then the above corollary shows that a subdirect product of two free groups is
$p$-conjugacy separable if and only if its index is a power of $p$ (see Corollary \ref{cor:crit_p-cs_for_fin_ind}). A result of Toinet \cite{Toinet} states that the subgroups of right angle Artin groups which are open in the pro-$p$ topology are $p$-conjugacy separable.
Since direct products of free groups are right angled Artin groups, Corollary \ref{cor:crit_for_c-cs_if_fin_ind} shows that the openness assumption in Toinet's theorem is indeed necessary and cannot be dropped (see Example~\ref{ex:ind-p}).

In Subsection \ref{subsec:necessity} we investigate a gap between the sufficient criterion for $\cC$-conjugacy separability of subdirect products of free groups provided
by Proposition \ref{prop:crit_of_CS_for_subdirect} and the necessary criterion given by Theorem \ref{thm:non-cs_for_subdir_of_hyp}. Theorem \ref{thm:near_converse} shows
that this gap is quite small. For some pseudovarieties of groups the gap does not exist at all (e.g., when $\cC=\cC_p$), though the general case is unclear:
see Question \ref{q:css_vs_res-C}.

In Sections \ref{sec:non_cs_fi_overgps} and \ref{sec:cs_overgroups} we use fibre products to
construct finitely presented groups demonstrating exotic
behaviour with respect to conjugacy separability. In these two sections we are concerned with the case when $\cC$ is the class of all finite groups. Thus
a \emph{hereditary conjugacy separable group} is a group where every subgroup of finite index is conjugacy separable.
The following statement is a special case of Theorem~\ref{thm:strong_vers}:

\begin{thm}\label{thm:strong_vers-simplified} There exists a finitely presented  hereditarily conjugacy separable group $G$ which has an
overgroup $K$ such that $|K:G|=2$, $K$ is not conjugacy separable and has unsolvable conjugacy problem.
\end{thm}

Note that every finite index subgroup of the group $G$ from Theorem \ref{thm:strong_vers-simplified} has solvable conjugacy problem by Mostowski's result \cite{Mostowski}.
First examples of finitely presented groups with solvable conjugacy problem having index $2$ overgroups with unsolvable conjugacy problem were
constructed by Collins and Miller in \cite{Col-Mil} and, independently, by Goryaga and Kirkinskii in \cite{Gor-Kirk} (conjugacy problem for
finite index subgroups in these examples was not investigated).

Theorem \ref{thm:strong_vers-simplified} shows that hereditary conjugacy separability is not stable under commensurability, which was previously unknown.
It is now natural to ask about the converse of Theorem~\ref{thm:strong_vers-simplified}:

\begin{question}\label{q:conv}
Does there exist a group $G$ such that every finite index overgroup $K$, of $G$, is conjugacy separable (resp. has solvable conjugacy problem),
but $G$ possesses a subgroup of finite index  which is not conjugacy separable (resp. has unsolvable conjugacy problem)?
\end{question}

Surprisingly, we discovered that the answers to both versions of Question \ref{q:conv} are negative. More precisely, in Corollary \ref{cor:ind_2->overind_3} (resp.
Corollary \ref{cor:ind_2->overind_3_CP}) we show that if a group $G$ has an index $2$ subgroup
$H$ such that $H$ is not conjugacy separable (resp. $H$ has unsolvable conjugacy problem), then there is an overgroup $K$, of $G$, with $|K:G|=3$ such that
$K$ is not conjugacy separable (resp. $K$ has unsolvable conjugacy problem). More generally, we use permutational wreath products to prove

\begin{cor}\label{cor:f_i_non-cs->f_oi_non-cs} If a group $G$ is not hereditarily conjugacy separable subgroup, then $G$ has an overgroup $K$, with $|K:G|<\infty$,
such that $K$ is not conjugacy separable.
\end{cor}

A similar statement about the conjugacy problem is given in Corollary \ref{cor:f_i_u_CP->f_oi_u_CP}. According to the construction in
Corollary \ref{cor:f_i_non-cs->f_oi_non-cs}, $G$ will not be normal in its overgroup $K$. This turns out to be the only obstruction.
The next result is a special case of Theorem \ref{thm:norm_overgps-cs}.

\begin{thm} \label{thm:norm_overgps-cs_simple} There exists a finitely presented group $G$, containing a subgroup $G'\lhd G $, of index $2$,
such that $G'$ is not conjugacy separable, but for every group
$K$, with $G \lhd K$ and $|K:G|<\infty$, $K$ is conjugacy separable.
\end{thm}

The first example of a conjugacy separable group with a non-conjugacy separable subgroup of finite index was constructed by Chagas and Zalesskii in
\cite{Chag-Zal-fp_res_free}; the first finitely  generated and finitely presented such examples were given in \cite{M-M}. However, these examples did not provide
any  information about conjugacy separability of finite index (normal) overgroups.

Our proofs of Theorems \ref{thm:strong_vers-simplified} and \ref{thm:norm_overgps-cs_simple} rely on a combination of fibre products with Rips-type constructions.
Such a combination was pioneered by Baumslag, Bridson, Miller and Short \cite{BBMS:1-2-3_sym}, who gave a sufficient criterion for finite presentability of symmetric
fibre products. To prove Theorem~\ref{thm:strong_vers-simplified} we modify the original construction of Rips \cite{Rips},
to ensure that the resulting small cancellation group admits an automorphism of finite order whose fixed subgroup projects onto any given
finitely generated subgroup of the original finitely presented group (see Proposition~\ref{prop:constr_of_autom_hyp}).
The proof of Theorem~\ref{thm:norm_overgps-cs_simple}, on the other hand, uses the Rips-type construction introduced by Bumagina and Wise in \cite{Bum-Wise}.
It allows one to minimize the group of automorphisms of the finitely generated normal subgroup in the resulting small cancellation group,
providing us with some control over the centralizers in any finite index normal overgroup of this group (see Lemma \ref{lem:centralizers_in_tilde_F}).

Conjugacy separability in both Theorems \ref{thm:strong_vers-simplified} and \ref{thm:norm_overgps-cs_simple} is established using the sufficient criterion
(Proposition \ref{prop:crit_of_CS_for_subdirect}) together with the fact that any group commensurable to a classical finitely presented
$C'(1/6)$ small cancellation group is hereditarily conjugacy separable.
This fact is a consequence of the results of Wise \cite{Wise-small_canc-cubical} and Agol \cite{Agol}, implying that such groups are virtually compact special
(in the sense of Haglund and Wise \cite{H-W}), and a theorem of the author and Zalesskii \cite{M-Z-vcs}, claiming that virtually compact special hyperbolic
groups are conjugacy separable.

The last Section \ref{sec:p} is devoted to the study of $p$-conjugacy separability of subdirect products.
In fact, our methods work more generally, when $\cC$
is a pseudovariety of $Q'$-groups, for some set of primes $Q$ (see Definition \ref{def:q'}). In this case we obtain a criterion which is both necessary
and sufficient for $\cC$-conjugacy separability of full subdirect products of non-abelian free groups (or of $\cC$-hereditarily conjugacy separable torsion-free
hyperbolic groups) -- see Theorem \ref{thm:q'-iff_crit}.
The following corollary is a special case of that criterion (recall that, for any prime $p$, a \emph{$p$-group} is a periodic group where every element has a $p$-power order).

\begin{cor}\label{cor:p_cs-iff_crit-simple}
Suppose that $p$ is a prime, $G \leqslant \FF$ is a full subdirect product of non-abelian free groups $F_1$, $F_2$, and $N_1\coloneq G \cap F_1$. Then
the following are equivalent:
\begin{itemize}
  \item[(i)] $G$ is $p$-conjugacy separable;
  \item[(ii)] $F_1/N_1$ is a residually finite $p$-group.
\end{itemize}
\end{cor}

Note that by a \emph{residually finite $p$-group} we mean a $p$-group which is residually finite.
Corollary \ref{cor:p_cs-iff_crit-simple} reveals an interesting and unexpected connection between $p$-conjugacy separability of the subdirect product $G$
and periodicity of the quotient $F_1/N_1$. There is no such connection in the case of standard conjugacy separability, when $\cC$ is the class of all finite groups
(cf. Example~\ref{ex:fibre-Z}).
Using the well-known fact that a subdirect product $G\leqslant \FF$, of finitely generated free groups $F_1$ and $F_2$,
is finitely generated if and only if the quotient $F_1/N_1$ is finitely presented, where $N_1 \coloneq G \cap F_1$, we prove

\begin{cor}\label{cor:equiv_to_ex_of_fp_p-gp} For each prime  $p$ the following statements are equivalent:
\begin{itemize}
  \item[(1)] there exists an infinite finitely presented residually finite $p$-group;
   \item[(2)] there exists a finitely generated $p$-conjugacy separable subgroup $G\leqslant H\times H$, where $H$ is the free group of rank $2$,
  such that $G$ is not virtually a direct product of two free groups;
  \item[(3)] there exists a finitely generated $p$-conjugacy separable full subdirect product $G \leqslant \FF$, where $F_1,F_2$ are
  free groups, such that   $|(\FF):G|=\infty$.
\end{itemize}
\end{cor}

The existence of an infinite finitely presented residually finite $p$-group is a long standing open problem, and Corollary \ref{cor:equiv_to_ex_of_fp_p-gp} shows that this
problem can be reformulated in terms of $p$-conjugacy separable subgroups in the direct product of two free groups. This can be considered as a further motivation for the study
of $p$-conjugacy separability.

Corollary \ref{cor:p_cs-iff_crit-simple} also shows that $p$-conjugacy separability of a subdirect product
is an extremely sensitive and rare condition. In particular, a proper full subdirect product of two non-abelian
free groups can be $p$-conjugacy separable for at most one prime $p$ (see Corollary \ref{cor:p-cs_for_2_p's-subdir}).
More generally, we obtain the following statement.

\begin{cor}\label{cor:p-cs_for_2_p's-general}
A subgroup $G$, of a direct product of two free groups, is $p$-conjugacy separable for at least two distinct primes $p$ if and only if $G$ is itself
isomorphic to a direct product of two free groups (one or both of which may be trivial).
\end{cor}

\section{Background}
\subsection{Notation}
Given a group $G$, a subgroup $H \leqslant G$ and an element $g \in G$, we will write $g^H \coloneq \{hgh^{-1} \mid h \in H\} \subseteq G$ to
denote the \emph{$H$-conjugacy class of $g$} and $\C_H(g) \coloneq \{h \in H \mid hgh^{-1}=g\}$ to denote the \emph{centralizer of $g$ in $H$}.
We will also let $\C_G(H) \coloneq \bigcap_{h \in H} \C_G(h)$ denote the \emph{centralizer of $H$ in $G$}.

Throughout the paper we will be working with direct products $F_1\times F_2$, so to simplify the notation, we will often
identify $F_1$ with the subgroup $F_1\times \{1\} =\{(f,1) \mid f \in F_1 \}\leqslant \FF$;
similarly, we will identify $F_2$ with  $\{1\}\times F_2 =\{(1,g) \mid g \in F_2\} \leqslant \FF$.

\subsection{Subdirect products}
The following statement summarizes basic properties of subdirect products.
\begin{lemma} \label{lem:norm_in_G->norm_in_F} Let $G\leqslant F_1\times F_2$ be a subdirect product of some groups $F_1,F_2$. Then
\begin{itemize}
  \item[(i)] for any normal subgroup $N \lhd G$ and any $i \in \{1,2\}$, the intersection $N \cap F_i$ is normal in $F_i$ and in $F_1\times F_2$. In particular, $N_i\coloneq G \cap F_i$ is normal in $F_i$, for each $i=1,2$, and in $G$.
  \item[(ii)] $F_1/N_1 \cong G/(N_1\times N_2) \cong F_2/N_2$.
  \item[(iii)] $|(\FF):G|<\infty$ if and only if $|F_1/N_1|<\infty$, in which case $|(\FF):G|=|F_1/N_1|$.
\end{itemize}

\end{lemma}

\begin{proof} (i) Any element of the intersection $N \cap F_1$ has the form $(h,1)$, for some $h \in F_1$.
Since $G \leqslant F_1\times F_2$ is subdirect, for each $f_1 \in F_1$, there
exists $f_2 \in F_2$ such that $(f_1,f_2) \in G$. Since $N\lhd G$, we have
\[ (f_1,1) (h,1) (f_1,1)^{-1} =(f_1,f_2) (h,1) (f_1,f_2)^{-1} \in  N.\]
On the other hand, obviously $(f_1,1) (h,1) (f_1,1)^{-1}=(f_1hf_1^{-1},1) \in F_1$,
hence $N \cap F_1$ is normal in $F_1$. Similarly one can show that $N\cap F_2$ is normal in $F_2$.
Evidently any normal subgroup of $F_i$ is also normal in the direct product $F_1\times F_2$, for $i=1,2$.

(ii) Let $\rho_i:F_1\times F_2 \to F_i$ be the natural projection, $i=1,2$. Then $\ker(\rho_2)=F_1$ and $\rho_2(G)=F_2$, as $G$ is subdirect.
Hence $F_2=\rho_2(G)\cong G/(G \cap F_1)=G/N_1$, and since $\rho_2(N_2)=N_2$, we get
$F_2/N_2 = \rho_2(G)/\rho_2(N_2) \cong G/(N_1\times N_2)$. Similarly, $F_1/N_1 \cong G/(N_1\times N_2)$, as required.

(iii) Clearly, if $|(\FF):G|<\infty$, then $|F_1/N_1|=|F_1:(F_1 \cap G)|\le |(\FF):G|<\infty$. On the other hand, if $F_1/N_1$ is finite,
then so is $F_2/N_2$ by claim (ii), hence
the quotient $(\FF)/(N_1\times N_2)\cong F_1/N_1\times F_2/N_2$ is finite as well. Since $N_1\times N_2 \subseteq G$, we can deduce that $|(\FF):G|<\infty$.

Now, assuming that $|F_1/N_1|<\infty$, we
have
\[|(\FF):G|=\left|\frac{\FF}{N_1\times N_2}:\frac{G}{N_1\times N_2}\right|=\frac{|F_1/N_1|\,|F_2/N_2|}{|F_1/N_1|}=|F_1/N_1|,\]
as  $|F_1/N_1|=|F_2/N_2|=|G/(N_1\times N_2)|$ by claim (ii).
\end{proof}

\subsection{Constructing subdirect products} \label{subsec:constr-subdir}
In this subsection we will review two main methods for constructing subdirect products of two groups.

Let $F_1,F_2, P$ be groups with epimorphisms $\psi_i:F_i \twoheadrightarrow P$, $i=1,2$. The \emph{fibre product of $F_1$ and $F_2$ corresponding to $\psi_1$ and $\psi_2$}
is defined as the subgroup $G \leqslant F_1 \times F_2$ given by
\begin{equation}\label{eq:fibre_prod-def}
G \coloneq \{(f_1,f_2) \in F_1 \times F_2 \mid \psi_1(f_1)=\psi_2(f_2)\}.
\end{equation}
If $F_1=F_2$ and $\psi_1=\psi_2$ then $G$ is said to be the \emph{symmetric fibre product of $F_1$ corresponding to $\psi_1$}.

The fibre product $G$, given by \eqref{eq:fibre_prod-def}, is always subdirect in $F_1\times F_2$ (because $\psi_1(F_1)=\psi_2(F_2)=P$).
Moreover, $N_i \coloneq G \cap F_i=\ker\psi_i$, so that $F_i/N_i \cong P$, $i=1,2$.

Conversely, it's not hard to show that if $G \leqslant F_1 \times F_2$ is a subdirect product of groups $F_1,F_2$, then $G$ is the fibre product of $F_1$ and $F_2$ with respect to some epimorphisms $\psi_i:F_i \to P$, where $P=G/(N_1\times N_2) \cong F_1/N_1 \cong F_2/N_2$, and $N_i \coloneq G \cap F_i= \ker\psi_i$, $i=1,2$.

Another standard method for constructing subdirect products is to start with any group $A$ which has two normal subgroups $L_i \lhd A$, and let
$F_i\coloneq A/L_i$, $i=1,2$, and $G\coloneq A/(L_1 \cap L_2)$. Then the map $\eta:G \to F_1 \times F_2$, $ a(L_1 \cap L_2) \stackrel{\eta}{\mapsto} (aL_1,aL_2)$,
gives rise to a natural subdirect embedding of $G$ into $F_1 \times F_2$. In particular, if $L_1 \cap L_2=\{1\}$, then $A=G$ is itself a subdirect product of $F_1$ and $F_2$.

Fibre products  have been used to construct numerous (counter-)examples in Group Theory. The first such construction is due to Miha{\u\i}lova \cite{Mih}, who
showed that the direct product of two non-abelian free groups contains a finitely generated subgroup with unsolvable membership problem. In fact,
Miha{\u\i}lova explicitly listed the finite generating set for her group, essentially proving the first claim of the following lemma.

\begin{lemma}[{cf. \cite[Prop. 3.28.2)]{Cornulier}}] \label{lem:Mih}
Let $G \leqslant F_1\times F_2$ be a subdirect product of groups $F_1$, $F_2$ and let $P \coloneq F_1/N_1$, where $N_1 \coloneq F_1 \cap G$.
\begin{itemize}
  \item[(a)] If the groups $F_1, F_2$ are finitely generated and $N_1$ is the normal closure of finitely many elements in $F_1$
  (which holds when $P$ is finitely presented) then $G$ is finitely generated.
  \item[(b)] If $G$ is finitely generated and $F_1$, $F_2$ are finitely presented then $P$ is finitely presented.
\end{itemize}
\end{lemma}

\begin{proof} (a)
Suppose that $F_1=\langle x_1,\dots,x_k\rangle$, $F_2=\langle v_1,\dots,v_l\rangle$, and
$N_1$ is the normal closure of finitely many elements $h_1,\dots h_m$ in $F_1$.
Since $G \leqslant \FF$ is subdirect, there exist $y_1,\dots,y_k \in F_2$ and $u_1,\dots,u_l \in F_1$ such that
$(x_1,y_1),\dots,(x_k,y_k) \in G$ and $(u_1,v_1),\dots,(u_l,v_l) \in G$. Since $G/N_1 \cong F_2$, it follows that
\[G=\langle (u_1,v_1),\dots,(u_l,v_l), N_1 \rangle.\]
By the assumptions, every $h \in N_1$ can be expressed as a product of conjugates of $h_1,\dots,h_m$ by elements from $\{x_1,\dots,x_k\}$, hence
$(h,1)$ is the product of the conjugates of $(h_1,1),\dots,(h_m,1)$ by elements from  $\{(x_1,y_1),\dots,(x_k,y_k)\}$. Consequently,
\[N_1 \subseteq \langle (x_1,y_1),\dots,(x_k,y_k),(h_1,1)\dots,(h_m,1) \rangle, \] which implies that
$G$ is generated by the elements $(u_1,v_1)$, $\dots$ ,$(u_l,v_l)$, $(x_1,y_1)$, $\dots$, $(x_k,y_k)$, $(h_1,1)$, $\dots$, $(h_m,1)$.

(b) Since $F_2 \cong G/N_1$ is finitely presented and $G$ is finitely generated,
$N_1$ is the normal closure of finitely many elements $(h_1,1),\dots,(h_m,1)$ in $G$. But the action of $G$ on $N_1$ by conjugation coincides with the action
of $F_1$ on $N_1$ by conjugation (because $(x,y) (h,1)(x,y)^{-1}=(xhx^{-1},1)$), hence $N_1$ is the normal closure of $h_1,\dots,h_m$ in $F_1$. And since $F_1$
is finitely presented, we can conclude that $F_1/N_1$ is also finitely presented.
\end{proof}

Lemma \ref{lem:Mih} provides a criterion for the subdirect product $G\in \FF$ to be finitely generated. The next \emph{asymmetric 1-2-3 theorem}, proved by
Dison \cite[Thm. 9.4]{Dison} (see also \cite[Thm. B]{BHMS:1-2-3_asym}), allows one to produce
finitely presented subdirect products (in the same spirit, Lemma \ref{lem:Mih} can be called the \emph{0-1-2 lemma}). The original \emph{symmetric 1-2-3
theorem} (when $F_1=F_2$) was proved by Baumslag, Bridson, Miller and Short \cite[Thm. B]{BBMS:1-2-3_sym}.

\begin{lemma}[{\cite[Thm. 9.4]{Dison}}]\label{lem:1-2-3}
Let $F_1,F_2$ be finitely presented groups and let $G \leqslant \FF$ be a subdirect product. If $N_1\coloneq G \cap F_1$ is
finitely generated and $P \coloneq F_1/N_1$ is of type $\mathrm{F}_3$ then $G$ is finitely presented.
\end{lemma}

The ``input'' for the 1-2-3 theorem can be very conveniently provided by the famous Rips's construction \cite{Rips} or by its numerous enhancements/modifications
(cf. \cite[Thm. 3.1]{BBMS:1-2-3_sym}, \cite[Thm.~3.1]{Wise-Rips}, \cite[Thm. 15]{Bum-Wise} or \cite[Thm. 10.1]{H-W}),
claiming that for every finitely presented group $P$
there is a hyperbolic group $F$ (usually with many ``nice'' properties)
and a normal subgroup $N\lhd F$ such that $N$ is finitely generated and $F/N \cong P$.
Thus, if $P$ is of type $\mathrm{F}_3$, then, by Lemma \ref{lem:1-2-3}, the corresponding symmetric fibre product $G \leqslant F \times F$ is finitely presented.

\subsection{Acylindrically hyperbolic groups}\label{subsec:acyl_hyp}
In this paper we will mostly work with subdirect products of groups acting on $\delta$-hyperbolic spaces in some controlled way. One of the most general classes of such groups
is the class of acylindrically hyperbolic groups, introduced  by  Osin in \cite{Osin-acyl}. It includes non-abelian free groups (of arbitrary rank) and non-elementary hyperbolic
groups (in the sense of Gromov). The reader is referred to \cite[Ch. III]{B-H} for the basic theory of hyperbolic spaces and groups.

Recall that a group $F$ is said to be \emph{elementary} if it possesses a cyclic subgroup of finite index. Following Osin \cite{Osin-acyl},
we will say that a group $F$ is \emph{acylindrically hyperbolic} if it is non-elementary and admits an acylindrical
cobounded isometric action on a hyperbolic metric space with unbounded orbits.
We will not define the notion of acylindricity of the action here, as we will only use properties of such groups that have already been
established elsewhere (the interested reader is referred to \cite{Osin-acyl} for the background). If $G$ is a hyperbolic group then it acts acylindrically and coboundedly
on any Cayley graph corresponding to a finite generating set. If $G$ splits as a free product of two non-trivial groups (e.g., if $G$ is free and non-abelian),
then it acts acylindrically and co-boundedly on the Bass-Serre tree corresponding to this splitting (cf. \cite[Lemma 4.2]{M-O_acyl_trees}).

Given a group $F$ acting by isometries on a $\delta$-hyperbolic metric space $(\mathcal S,d)$, an element $f \in F$ is said to be \emph{loxodromic} if for some point
$s \in \mathcal{S}$ the function $\Z \to \mathcal{S}$, $n \mapsto f^n(s)$, is a quasi-isometric embedding; in particular, the order of $f$ must be infinite.
If the action of $F$ on $\mathcal S$ is acylindrical then every
loxodromic element $f \in F$ satisfies the \emph{WPD condition} of Bestvina and Fujiwara \cite{B-F}  -- see \cite[Def. 2.5]{Osin-acyl}.
It follows from the work of Dahmani, Guirardel and Osin \cite[Lemma~6.5, Cor.~6.6]{DGO} that such $f$ is contained in a
\emph{unique maximal elementary subgroup} $\E_F(f)$, and
\begin{equation}\label{eq:max_elem_descr}
\E_F(f)=\{g \in F \mid gf^ng^{-1}=f^{\pm n} \mbox{ for some }  n \in \N\}.
\end{equation}

Now, if $F$ acts on a $\delta$-hyperbolic metric space $\mathcal{S}$ coboundedly and
$H \leqslant F$ is a non-elementary subgroup containing at least one loxodromic element
then, by \cite[Lemma 5.6]{A-M-S},  there is a largest finite subgroup $\E_F(H) \leqslant F$, normalized by $H$.
In particular, $F$ itself has a maximal finite normal subgroup $\E_F(F)$ (cf. \cite[Thm. 2.24]{DGO}).

\begin{lemma}\label{lem:inf_norm_sbgp->non-elem}  Suppose that $F$ is an acylindrically hyperbolic group without non-trivial finite normal subgroups.
Then any non-trivial normal subgroup $N \lhd F$  is non-elementary, satisfies $\E_F(N)=\{1\}$, and there is an element
$h \in N$ such that $\C_F(h)=\langle h \rangle \subseteq N$. Moreover, for any $x \in F $ there exists $n \in \Z$ such that
the element $f \coloneq h^n x \in F$  satisfies $\C_F(f)=\langle f \rangle$.
\end{lemma}

\begin{proof} Fix some non-elementary cobounded acylindrical action of $F$ on a $\delta$-hyperbolic metric space $\mathcal S$. Since $N \neq \{1\}$ is normal in $F$,
it must be infinite by the assumptions, so, according to \cite[Lemma 7.2]{Osin-acyl},
$N$ is non-elementary and has at least one loxodromic element. Therefore we can consider $\E_F(N)$, the maximal
finite subgroup of $F$, normalized by $N$.

Now, let us show that $\E_F(N)=\E_F(F)$. Indeed, $\E_F(F) \subseteq \E_F(N)$, as $\E_F(F)$ is finite and is normalized by $N$, and $\E_F(N)$ is the largest subgroup
of $F$ with this property. On the other hand, it is easy to check that since $N \lhd F$, $\E_F(N) \lhd F$ (because for each $f \in F$, $f\E_F(N)f^{-1}$ is also normalized by $N$), hence $\E_F(N) \subseteq \E_F(F)$. Thus we can conclude that $\E_F(N)=\E_F(F)$, but $\E_F(F)=\{1\}$ as $F$ contains no non-trivial finite normal subgroups, therefore
$\E_F(N)=\{1\}$.

We can now apply \cite[Lemma 5.12]{A-M-S} to find a loxodromic element $h \in N$ such that $\E_F(h)=\langle h \rangle \E_F(N)=\langle h \rangle$. Since
$\C_F(h) \subseteq \E_F(h)$ by \eqref{eq:max_elem_descr}, we deduce that $\C_F(h)=\langle h \rangle$, as required.

If $x \notin \langle h \rangle=\E_F(h)$, then the final claim of the lemma follows from \cite[Lemma 5.13]{A-M-S}. Otherwise, if $x=h^m$ for some $m \in \Z$, then we choose
$n=1-m \in \Z$, so that $f=h^n x=h$, and the required equality $\C_F(f)=\E_F(f)=\langle f \rangle$ holds for $f$ because it holds for $h$.
\end{proof}

\section{The pro-\texorpdfstring{$\cC$}{C} topology on subdirect products}
Our main tool for studying $\cC$-conjugacy separability is the pro-$\cC$ topology. In this section we will investigate some basic properties of this topology
on subdirect products of two groups.

\subsection{Pseudovarieties and pro-\texorpdfstring{$\cC$}{C} topology}
We will say that $\cC$ is a \emph{pseudovariety of groups} if $\cC$ is a class of groups closed under isomorphisms, subgroups, direct products and quotients.
In other words, to be a pseudovariety $\cC$ must satisfy the following conditions:

\begin{itemize}
  \item if $A \in \cC$ and $B \cong A$ then $B \in \cC$;
  \item if $A \in \cC$ and $B \leqslant A$ then $B \in \cC$;
  \item if $A,B \in \cC$ then $A \times B \in \cC$;
  \item if $A \in \cC$ and $N \lhd A$ then $A/N \in \cC$.
\end{itemize}

By a \emph{pseudovariety of finite groups} we will mean a pseudovariety $\cC$ such that each member of $\cC$ is a finite group. For example, the class of all finite  groups
or the class of all finite $p$-groups, for some prime $p$, are pseudovarieties of finite groups.

Let $\cC$ be a pseudovariety of groups.
Given a group $G$ and $N\lhd G$, $N$ is said to be a {\it co-$\cC$ subgroup} of $G$ if $G/N \in \cC$.
One can define the \emph{pro-$\cC$ topology} on any group $G$ by letting the cosets of co-$\cC$ subgroups be the basic open sets.
Since $\cC$ is closed under taking subgroups, it is easy to see that any group homomorphism $G \to G_1$ is continuous with respect to the pro-$\cC$ topologies on $G$ and $G_1$
(see \cite[pp. 8--11]{Ferov-thesis} for background on pro-$\cC$ topologies). In particular, we can make the following observation.

\begin{rem}\label{rem:pro-C_on_sbgp} If $H$ is any subgroup of a group $G$ and $Y \subset G$ is closed in the \proc on $G$ then $H \cap Y$ is closed in the \proc on $H$.
\end{rem}

We will say that a subset $X \subseteq G$ is \emph{$\cC$-closed} (respectively, \emph{$\cC$-open}) if $X$ is closed (respectively, open) in the \proc on $G$.
Since $G$, equipped with its pro-$\cC$ topology, is a topological group, for any $\cC$-closed (respectively, $\cC$-open) subset $X \subset G$, and any $g \in G$,
$gX$ and $Xg$ are also $\cC$-closed (respectively, $\cC$-open) in $G$. Since the complement of a subgroup is a union of cosets modulo this subgroup,
the complement of a $\cC$-open subgroup is itself $\cC$-open, so that the subgroup is $\cC$-closed. More generally, the following holds:

\begin{rem}[{cf. \cite[Lemma 2.6 on p. 10]{Ferov-thesis}}]\label{rem:open->closed}
If $G$ is a group and $H\leqslant G$ is a subgroup containing a $\cC$-open subgroup then $H$ is both $\cC$-open and $\cC$-closed in $G$, and $H$ contains some
co-$\cC$ subgroup of $G$.
\end{rem}

\begin{lemma}\label{lem:full_preimage} Suppose that $G$ and $P$ are groups and $\psi:G \to P$ is an epimorphism.
\begin{itemize}
  \item[(i)] For every $\cC$-open subset $X$ of $G$,
$\psi(X)$ is $\cC$-open in $P$ (in other words, $\psi$ is an open map when $G$ and $P$ are equipped with their pro-$\cC$ topologies).
  \item[(ii)] The \proc on $P$ coincides with the quotient topology induced by $\psi$ from the \proc on $G$.
  \item[(iii)] Given any subset $Y \subseteq P$, $Y$ is $\cC$-closed in $P$ if and only if  $\psi^{-1}(Y)$ is $\cC$-closed in $G$.
\end{itemize}
\end{lemma}

\begin{proof} (i) Let $N \coloneq \ker \psi$.
Clearly, to show that $\psi$ is an open map, it is enough to prove that $\psi(H)$ is $\cC$-open in $P$ for every co-$\cC$ subgroup $H \lhd G$.
This is indeed the case because $P/\psi(H) \cong G/(HN)$ can be viewed as a quotient of $G/H \in \cC$, and $\cC$ is closed under taking quotients by the assumptions.

Claim (ii) is a direct consequence of claim (i), as any open continuous map is a quotient map (cf. \cite[Sec. 2-11]{Munkres}).
Claim (iii) is, of course, simply a restatement of claim (ii).
\end{proof}

The pro-$\cC$ topology on $G$ is Hausdorff if and only if $G$ is \emph{residually-$\cC$}, i.e., for every $g \in G\setminus\{1\}$ there is a group $M \in \cC$ and a
homomorphism $\varphi:G \to M$ such that $\varphi(g) \neq 1$ in $M$. This is also equivalent to the statement that the singleton $\{1\}$ is $\cC$-closed in $G$.

If $H$ is a subgroup of a group $G$, we will say that the \proc on $H$ is a \emph{restriction of the \proc topology on $G$} if every  subset $X \subseteq H$, closed
in the \proc on $H$,  is also closed in the \proc on $G$ (in particular, $H$ must be $\cC$-closed in $G$); in other words, the converse of the claim of
Remark~\ref{rem:pro-C_on_sbgp} holds.

A pseudovariety of groups $\cC$ is said to be \emph{extension-closed} if for any group $G$, containing a normal subgroup
$N\lhd G$ such that $N \in \cC$ and $G/N\in \cC$, one has $G \in \cC$.

\begin{lemma}\label{lem:restr_crit} Let $\cC$ be a pseudovariety of finite groups. Suppose that $H$ is a subgroup of a group $G$.
\begin{itemize}
  \item[(a)] The \proc on $H$ is a restriction of the \proc on $G$ if and only if every co-$\cC$ subgroup of $H$ is $\cC$-closed in $G$.
  \item[(b)] If the pseudovariety $\cC$ is extension-closed and $H$ is $\cC$-open in $G$ then the \proc on $H$ is the restriction of the \proc on $G$.
\end{itemize}
\end{lemma}

\begin{proof} (a) The necessity is clear by Remark \ref{rem:open->closed}. To prove the sufficiency, assume that every co-$\cC$ subgroup of $H$ is $\cC$-closed in $G$.
Evidently, to prove that every $\cC$-closed subset of $H$ is $\cC$-closed in $G$ it is enough to show this for each basic $\cC$-closed subset $X$ in $H$.
Thus $X=\bigcup_{i \in I} N h_i$, for some co-$\cC$ subgroup $N$ of $H$, and some $h_i \in H$, $i \in I$. Since $H/N \in \cC$ and $\cC$ consists of finite groups by
the assumption, we can deduce that $|H:N|<\infty$, so that $X$ is a finite union of cosets modulo $N$. Consequently, $X$ is $\cC$-closed in $G$ as a finite union of
$\cC$-closed sets, because each coset $N h_i$ is $\cC$-closed in $G$.

(b) See the proof of \cite[Lemma 3.1.4.(a)]{RZ}.
\end{proof}

\subsection{Subdirect products}
\begin{rem}\label{rem:YG_int_F_1} Suppose that $G \leqslant \FF$ is a subdirect product, $N_1 \coloneq G \cap F_1$ and $Y \subseteq F_1$ is any subset. Then
$F_1 \cap YG=YN_1$. (Indeed, since $Y \subseteq F_1$,  $F_1 \cap YG=Y(F_1\cap G)=YN_1$.)
\end{rem}

The next lemma will be used throughout the paper.

\begin{lemma} \label{lem:crit_closed}  Let $F_1,F_2$ be groups and let $G \leqslant F_1 \times F_2$ be a subdirect product. If $\cC$ is a pseudovariety of groups
and $X \subseteq F_1$ is any subset, then the product $XG=\{(x,1) g \mid x \in X,~g \in G\}\subseteq F_1\times F_2$ is closed in the \proc on $F_1\times F_2$
if and only if $XN_1$ is closed in the \proc on $F_1$, where $N_1\coloneq F_1 \cap G$.
\end{lemma}

\begin{proof}
If $XG$ is $\cC$-closed in $F_1\times F_2$, then $XN_1=XG \cap F_1$ (cf. Remark \ref{rem:YG_int_F_1})
is $\cC$-closed in $F_1$ by Remark~\ref{rem:pro-C_on_sbgp}.

Now, suppose that $XN_1$ is $\cC$-closed in $F_1$.
Let $\mathcal{N}_{\cC}(F_i)$ denote the set of all co-$\cC$ subgroups of $F_i$, $i=1,2$. First, let us show that
\begin{equation}\label{eq:AXG}
XG=\bigcap_{A \in \mathcal{N}_{\cC}(F_1)} AXG \quad\mbox{in } F_1\times F_2.
\end{equation}

Indeed, evidently, the right-hand side contains the left-hand side. For the opposite inclusion, assume that $(f_1,f_2) \in AXG$, for some $f_1 \in F_1$, $f_2 \in F_2$
and for every $A \in \mathcal{N}_{\cC}(F_1)$. Then, since $G$ is subdirect in $F_1\times F_2$, there is $h_1 \in F_1$ such that $(h_1,f_2) \in G$, hence
$(f_1h_1^{-1},1) \in AXG$ for all $A \in \mathcal{N}_{\cC}(F_1)$. The latter, combined with Remark \ref{rem:YG_int_F_1}, yields that
\begin{equation}\label{eq:fxgh}
(f_1h_1^{-1},1) \in AXG \cap F_1 =AXN_1, \mbox{ for all } A \in \mathcal{N}_{\cC}(F_1).
\end{equation}
 However, since $XN_1$ is $\cC$-closed in $F_1$, for every $(h,1) \in F_1 \setminus XN_1$ there is $A_h \in\mathcal{N}_{\cC}(F_1)$
such that $A_h(h,1) \cap XN_1=\emptyset$, which is equivalent to $(h,1) \notin A_hXN_1$. Therefore, in view of \eqref{eq:fxgh}, we can conclude that
$(f_1h_1^{-1},1) \in XN_1$, consequently $(f_1,f_2)=(f_1h_1^{-1},1)(h_1,f_2) \in XN_1 G=XG$. Thus we have shown that $\bigcap_{A \in \mathcal{N}_{\cC}(F_1)} AXG \subseteq XG$,
so \eqref{eq:AXG} holds.

Thus it remains to prove that for each $A  \in \mathcal{N}_{\cC}(F_1)$ the subset $AXG$ is $\cC$-closed in $F_1\times F_2$.
Take any $A \in \mathcal{N}_{\cC}(F_1)$, and note that $P\coloneq AG$ is a subgroup of $F_1\times F_2$
(because $A\lhd F_1$, and so $A \lhd F_1\times F_2$), and it is subdirect as it contains $G$. Combining Lemma \ref{lem:norm_in_G->norm_in_F}.(ii) with
Remark~\ref{rem:YG_int_F_1}, we obtain
\[ F_2/(F_2 \cap P) \cong F_1/(F_1 \cap P)=F_1/(AN_1). \]
It follows that $F_2/(F_2 \cap P) \in \cC$, as it is a quotient of the group $F_1/A \in \cC$ and $\cC$ is closed under taking quotients. Thus
$B\coloneq F_2 \cap P \in \mathcal{N}_{\cC}(F_2)$, and so $(F_1 \times F_2)/(A\times B) \cong F_1/A \times F_2/B \in \cC$, i.e., $A\times B$ is a
co-$\cC$ subgroup of $F_1\times F_2$.

We have shown that $P=AG$ contains the $\cC$-open subgroup $A \times B$ of $F_1\times F_2$, hence it is itself $\cC$-open by Remark \ref{rem:open->closed}.
Observe that $AXG=XAG=XP$ is a union of cosets modulo $P$, so its complement in $F_1\times F_2$ is also a union of cosets modulo $P$, hence this complement is
$\cC$-open, so that $AXG$ is $\cC$-closed in $F_1\times F_2$, as claimed.

Now, \eqref{eq:AXG} implies that $XG$ is $\cC$-closed in $F_1 \times F_2$, and the lemma is proved.
\end{proof}

\begin{cor}\label{cor:subdir-closed-crit} Suppose that $G \leqslant F_1\times F_2$ is a subdirect product of groups $F_1$ and $F_2$, and $\cC$ is a pseudovariety of groups.
Then the following are equivalent:
\begin{enumerate}
  \item[(i)] $G$ is $\cC$-closed in $F_1\times F_2$;
  \item[(ii)] $N_1\coloneq G \cap F_1$ is $\cC$-closed in $F_1$;
  \item[(iii)] $F_1/N_1$ is residually-$\cC$.
\end{enumerate}
\end{cor}

\begin{proof} The equivalence of (i) and (ii) follows from Lemma \ref{lem:crit_closed} (set $X=\{1\}$), and the equivalence of (ii) and (iii) is given by the standard fact (left
as an exercise for the reader) that a normal subgroup $N$ of a group $F$ is $\cC$-closed if and only if the quotient $F/N$ is residually-$\cC$.
\end{proof}

We will now aim to give a sufficient criterion for the \proc on a subdirect product to be a restriction of the \proc on the direct product. To this end we will need the following definition.

\begin{defn}\label{defn:highly_r-C} Let $M$ be a group and let $\cC$ be a pseudovariety of groups. We will say that $M$ is \emph{highly residually-$\cC$} if for every group $L$ and every $K \in \cC$  fitting in the short exact sequence
 \[\{1\} \to K \to L \to M \to \{1\},\]
 $L$ is residually-$\cC$.
\end{defn}
In other words, $M$ is highly residually-$\cC$ if each extension of a $\cC$-group by $M$ is residually-$\cC$. Our terminology is motivated by that of
Lorensen \cite[p. 1710]{Lor}, where he calls a group $M$ \emph{highly residually finite} if each (finite-by-$M$) group is residually finite.

\begin{lemma} \label{lem:restr_on_subdir} Let $F_1,F_2$ be groups, let $G \leqslant F_1\times F_2$ be a subdirect product, and let $\cC$ be an extension-closed pseudovariety of finite groups. If $F_1/N_1$ is highly residually-$\cC$, for $N_1\coloneq F_1 \cap G$, then the \proc on $G$ is a restriction of the \proc on $F_1\times F_2$.
\end{lemma}

\begin{proof} In view of Lemma \ref{lem:restr_crit}.(a), it is enough to show that every co-$\cC$ subgroup $G'$ of $G$ is $\cC$-closed in $F_1\times F_2$.
Let $\rho_i:F_1\times F_2 \to F_i$ denote the canonical projection, and set $F_i'=\rho_i(G')$, $i=1,2$. Then $F_i' \lhd F_i$, $i=1,2$, as $G' \lhd G$, and
$F_1/F_1'=\rho_1(G)/\rho_1(G')\cong G/(G'N_2)$, where $N_2\coloneq G \cap\ker(\rho_1)=G \cap F_2$. Thus $F_1/F_1' \in \cC$, as a quotient of the group $G/G' \in \cC$;
similarly, $F_2/F_2' \in \cC$.
Therefore, $(\FF)/(F_1'\times F_2') \cong F_1/F_1'\times F_2/F_2' \in \cC$; in particular, $F_1'\times F_2'$ is a $\cC$-open subgroup of $\FF$.

Obviously $G' \leqslant F_1'\times F_2'$ is a subdirect product, by construction. Now, observe that
\[N_1'\coloneq G' \cap F_1'=G' \cap F_1=G' \cap N_1\] is a normal subgroup of $F_1$ by
Lemma~\ref{lem:norm_in_G->norm_in_F}.(i),
and $N_1/N_1' \cong G'N_1/G' \leqslant G/G' \in \cC$, which yields that $N_1/N_1' \in \cC$ as $\cC$ is closed under taking subgroups. Recalling that $F_1/N_1$ is highly residually-$\cC$, the short exact sequence
\begin{equation}\label{eq:F/N_1'}
\{1\} \to N_1/N_1' \to F_1/N_1' \to F_1/N_1 \to \{1\},
\end{equation}
implies that $F_1/N_1'$ is residually-$\cC$, and hence its subgroup $F_1'/N_1'$ is residually-$\cC$ as well.

It remains to conclude that $G'$ is $\cC$-closed in $F_1'\times F_2'$ by Corollary \ref{cor:subdir-closed-crit}, which, in view of Lemma \ref{lem:restr_crit}.(b),
implies that $G'$ is $\cC$-closed in $\FF$, as required.
\end{proof}

\subsection{Cyclic subgroup separability}
Given a pseudovariety of groups $\cC$, a group $M$ is said to be \emph{cyclic subgroup $\cC$-separable} if every cyclic subgroup  is closed in the \proc on $M$.
As usual, if $\cC$ is the class of all finite groups, then we will simply write that $M$ is cyclic subgroup separable.

Since the trivial subgroup is cyclic, any cyclic subgroup $\cC$-separable group is residually-$\cC$. The converse is not true in general; for example, the metabelian
Baumslag-Solitar group $BS(1,2)\coloneq\langle a,t \;\|\; tat^{-1}=a^2\rangle$ is residually finite (\cite[Thm. 1]{Hall})
 but the cyclic subgroup $\langle a \rangle$ is not closed in the profinite topology on $BS(1,2)$, as it is conjugate to a proper subgroup of itself.

In Proposition \ref{prop:crit_of_CS_for_subdirect} below we will see that
cyclic subgroup $\cC$-separability of the quotient $F_1/N_1$ is important in proving that a subdirect product $G \leqslant \FF$ is $\cC$-conjugacy separable. In this subsection
we will discuss some permanence properties related to cyclic subgroup $\cC$-separability, that will be useful later on.

\begin{lemma}\label{lem:cyc-sep-1} Let $\cC$ be an extension-closed pseudovariety of finite groups and let $F$ be any group.
\begin{itemize}
  \item[(i)] If $F$ is cyclic subgroup $\cC$-separable then so is every subgroup $H \leqslant F$.
  \item[(ii)] If some $\cC$-open subgroup $H \leqslant F$ is cyclic subgroup $\cC$-separable then so is $F$.
  \item[(iii)] If $F$ is cyclic subgroup $\cC$-separable and $K \lhd F$ is a finite normal subgroup then $F/K$ is cyclic subgroup $\cC$-separable.
  \item[(iv)] Suppose that $K \lhd F$, $K \in \cC$, $F/K$ is highly residually-$\cC$ and cyclic subgroup $\cC$-separable.
  Then $F$ is  itself cyclic subgroup $\cC$-separable.
\end{itemize}
\end{lemma}

\begin{proof} Claim (i) is a trivial consequence of the definition and Remark \ref{rem:pro-C_on_sbgp}.

To prove claim (ii), suppose that $H$ is cyclic subgroup $\cC$-separable and $C \leqslant F$ is any cyclic subgroup.
Then the cyclic subgroup $C'\coloneq C \cap H$ is closed in the \proc on $H$, and Lemma \ref{lem:restr_crit}.(b) implies that $C'$ is also $\cC$-closed in $F$.
Now, $|C:C'| \le |F:H|<\infty$ as $\cC$ consists of finite groups and $H$ is $\cC$-open in $F$ (thus $H$ contains some co-$\cC$ subgroup of $F$ by
Remark~\ref{rem:open->closed}). Therefore $C$ is a finite union of cosets modulo $C'$, hence it is also $\cC$-closed in $F$.

To establish claim (iii) assume that $F$ is cyclic subgroup $\cC$-separable and $K \lhd F$, $|K|<\infty$. Then the cyclic subgroup $\{1\}$ is $\cC$-closed in $F$,
so, since $K$ is finite, there exists a co-$\cC$ subgroup $H \lhd F$ such that $K \cap H=\{1\}$.
It follows that the image $HK/K\cong H/(H\cap K)$, of $H$ in $F/K$, is naturally isomorphic to $H$ and is a co-$\cC$-subgroup of $F/K$ because $\cC$ is closed under
taking quotients.
Now, according to claim (i), $H$ is cyclic subgroup $\cC$-separable, hence so is $F/K$ by claim (ii).

It remains to prove claim (iv). By the assumptions, $F/K$ is highly residually-$\cC$, so $F$ is residually-$\cC$, hence there exists a co-$\cC$ subgroup $H \lhd F$
such that $H \cap K=\{1\}$ ($|K|<\infty$ as $K \in \cC$). As before, $H$ is isomorphic to its image $HK/K$ in $F/K$, hence it is cyclic subgroup $\cC$-separable by claim (i)
as  $F/K$ is cyclic subgroup $\cC$-separable. Therefore, in view of claim (ii), we can conclude that $F$ is cyclic subgroup $\cC$-separable.
\end{proof}

The next lemma will be useful for showing that every $\cC$-open subgroup of a subdirect product is $\cC$-conjugacy separable.

\begin{lemma} \label{lem:C-open-subdir} Let $\cC$ be an extension-closed pseudovariety of finite groups, let $G \leqslant \FF$ be a subdirect product of some
groups $F_1$, $F_2$ and let $N_1\coloneq G \cap F_1$. Suppose that $F_1/N_1$ is highly residually-$\cC$ and cyclic subgroup $\cC$-separable.
If $H\leqslant G$ is any $\cC$-open subgroup then $H$ is a subdirect product in
$J_1\times J_2$, for some $\cC$-open subgroups $J_i$ of $F_i$, $i=1,2$, and $J_1/( H \cap J_1)$ is cyclic subgroup $\cC$-separable.
\end{lemma}

\begin{proof} Naturally, we let $J_i\leqslant F_i$ be the image of $H$ under the projection to the $i$-th coordinate group, $i=1,2$.
Then $H \leqslant J_1\times J_2$ is subdirect, by construction. Now, by Remark~\ref{rem:open->closed}, $H$ contains some co-$\cC$ subgroup $G'$ of $G$. Using the same
notation as in Lemma~\ref{lem:restr_on_subdir}, let $F_i'\leqslant F_i$ denote the projection of $G'$ to the $i$-th coordinate group, $i=1,2$.
Then $F_i' \subseteq J_i$ and $F_i'$ is a co-$\cC$ subgroup of $F_i$, hence $J_i$ is $\cC$-open in $F_i$, $i=1,2$.

Now, if we let $N_1' \coloneq F_1 \cap G'$, then, by the argument from the proof of Lemma \ref{lem:restr_on_subdir}, $N_1' \lhd F_1$, $N_1/N_1'  \in \cC$
and the quotient $F_1/N_1'$ fits into the short exact sequence \eqref{eq:F/N_1'}. Therefore,
in view of Lemma \ref{lem:cyc-sep-1}.(iv), our assumptions on $F_1/N_1$ imply that $F_1/N_1'$ is cyclic subgroup $\cC$-separable.

Clearly, $N_1'=G' \cap F_1 \subseteq H \cap F_1=H \cap J_1 \subseteq G \cap F_1=N_1$, so the group $J_1/( H \cap J_1 )$ is isomorphic to the quotient of the group
$J_1/N_1'\leqslant F_1/N_1'$ by the finite normal subgroup $(H\cap J_1)/N_1' \leqslant N_1/N_1'$. Thus we can apply claims (i) and (iii) of Lemma  \ref{lem:cyc-sep-1} to
deduce that  the group $J_1/(H \cap J_1)$ is cyclic subgroup $\cC$-separable, as required.
\end{proof}

\section{Conjugacy separability of subdirect products} \label{sec:crit}
In this section we will give necessary and sufficient criteria for $\cC$-conjugacy separability of subdirect products of two groups.

\begin{rem}\label{rem:cs<=>each_cc_closed}
Observe that a group $G$ is $\cC$-conjugacy separable if and only if the $G$-conjugacy class $g^G$ is $\cC$-closed in $G$, for each $g \in G$.
\end{rem}

We will say that a pseudovariety of groups is \emph{non-trivial} if it contains at least one non-trivial group.
Basic examples of $\cC$-conjugacy separable groups are free groups:

\begin{lemma}\label{lem:free_are_C-hcs} Suppose that $\cC$ is a non-trivial extension-closed pseudovariety of groups and $F$ is a free group of arbitrary rank.
Then $F$ is $\cC$-hereditarily conjugacy separable.
\end{lemma}

\begin{proof} Since any subgroup of $F$ is also free, it is enough to show that $F$ is $\cC$-conjugacy separable.

By the assumptions, the class $\cC$ is closed under taking subgroups and contains at least one non-trivial group, so it must contain some non-trivial cyclic group,
and since $\cC$ is closed under quotients we deduce that $\Z/p\Z \in \cC$, for some prime $p$. Now, every non-trivial finite $p$-group $P$ has a normal series where
the sections are cyclic groups of order $p$, hence $P \in \cC$, as $\cC$ is closed under taking extensions.  Therefore $\cC$
contains the class $\cC_p$, of all finite $p$-groups. Since free groups are well-known to be $\cC_p$-conjugacy separable (cf. \cite[Prop. 5]{Rem}), we can conclude that
$F$ is $\cC$-conjugacy separable.
\end{proof}

In \cite[Thm. 1.2]{Ferov-1} Ferov proved that for an extension-closed pseudovariety of finite groups, any graph product of $\cC$-hereditarily conjugacy separable groups is
also $\cC$-hereditarily conjugacy separable. For our purposes we will need a much easier special case:

\begin{lemma}[{\cite[Lemma 4.2 on p. 18]{Ferov-thesis}}]\label{lem:C-cs-for_products} If $\cC$ is an extension-closed pseudovariety of finite groups, then the direct product of two
$\cC$-hereditarily conjugacy separable groups is $\cC$-hereditarily conjugacy separable.
\end{lemma}

In the case when $\cC$ is the class of all finite groups, Lemma \ref{lem:C-cs-for_products} was originally proved by Martino and the first author in \cite[Lemma 7.3]{M-M}.

\subsection{Criteria for conjugacy separability}\label{subsec:crit_for_cs}
Throughout this subsection we will assume that $\cC$ is an extension-closed pseudovariety of finite groups.

The following general criterion was proved by Ferov in \cite[Cor. 4.7 and Thm. 4.2]{Ferov-1}, it naturally extends the criterion found by the author in
\cite[Cor. 3.5 and Prop. 3.2]{M-RAAG}.
\begin{lemma} \label{lem:crit_for_CS} Let
$F$ be a $\cC$-hereditarily conjugacy separable group and let $G \leqslant F$ be a subgroup. Suppose that for each  $g \in G$
the double coset $\C_F(g) G$ is $\cC$-closed in $F$.  Then $g^G$ is $\cC$-closed in $F$ for each $g \in G$, in particular, $G$ is $\cC$-conjugacy separable.
\end{lemma}

We will say that a group $F$ \emph{has cyclic centralizers} if the centralizer $\C_F(f)$ is cyclic for each $f \in F\setminus\{1\}$.
Basic examples of groups with cyclic centralizers are free groups (\cite[Prop. 2.19 in Sec. I.2]{L-S}),
torsion-free hyperbolic groups \cite[Cor. 3.10 in Ch. III.$\Gamma$]{B-H}, $1$-relator groups with torsion \cite[Thm. 2]{Newman} and
$C'(1/6)$ small cancellation groups (\cite{Truf}).

The following proposition generalizes \cite[Prop. 7.5]{M-M}.

\begin{prop} \label{prop:crit_of_CS_for_subdirect} Suppose that $\cC$ is an extension-closed pseudovariety of finite groups.
Let $F_1$, $F_2$ be $\cC$-hereditarily conjugacy separable
groups with cyclic centralizers, let $G \leqslant \FF$ be a subdirect product and let
$N_1\coloneq G \cap F_1$. If $F_1/N_1$ is cyclic subgroup $\cC$-separable then $G$ is $\cC$-conjugacy separable.
\end{prop}

\begin{proof} We will aim to apply the criterion from Lemma \ref{lem:crit_for_CS}. So, consider any element $(g_1,g_2) \in G$. If $g_1=1$ in $F_1$ (or $g_2 = 1$ in $F_2$),
then $\C_{\FF}((g_1,g_2))$ contains all of $F_1$ (or all of $F_2$), and since $G\leqslant \FF$ is subdirect, we would have $\C_{\FF}((g_1,g_2))G=\FF$,
which is evidently $\cC$-closed in $\FF$.

Thus we can suppose that $g_1 \neq 1$ and $g_2 \neq 1$. Then $\C_{F_i}(g_i)=\langle f_i \rangle$, for some $f_i \in F_i$, by the assumptions, and $g_i=f_i^{m_i}$ for some
$m_i \in \Z\setminus\{0\}$, $i=1,2$. Therefore $\C_{\FF}((g_1,g_2))=\langle (f_1,1),(1,f_2) \rangle \cong \langle f_1\rangle \times \langle f_2 \rangle$, and, so
the subgroup $H\coloneq \langle (g_1,1),(1,g_2)\rangle=\langle (g_1,1),(g_1,g_2) \rangle$ has finite index in $\C_{\FF}((g_1,g_2))$. Note that $HG=\langle (g_1,1) \rangle G$
because $(g_1,g_2) \in G$ commutes with $(g_1,1)$. Thus, for any transversal $(a_1,b_1), \dots , (a_k,b_k)$ for the left cosets in $\C_{\FF}((g_1,g_2))/H$, we have
\begin{equation}\label{eq:C_FF(g1g2)G}
\C_{\FF}((g_1,g_2))G=\bigcup_{j=1}^k (a_j,b_j) H G=\bigcup_{j=1}^k (a_j,b_j) \langle (g_1,1) \rangle G.
\end{equation}
Now, recall that, by Lemma \ref{lem:crit_closed}, the double coset $\langle (g_1,1) \rangle G$ is $\cC$-closed in $\FF$,
provided the double coset $\langle g_1 \rangle N_1$ is $\cC$-closed in $F_1$.
Since $N_1 \lhd F_1$, in view of Lemma \ref{lem:full_preimage}.(iii)
the latter is equivalent to saying that the cyclic subgroup $\psi(\langle g_1\rangle)$ is $\cC$-closed in $F_1/N_1$, which is true by our assumptions,
where $\psi:F_1 \to F_1/N_1$ is the natural homomorphism.
Thus we can deduce that $\langle (g_1,1) \rangle G$ is $\cC$-closed in $\FF$, and so
\eqref{eq:C_FF(g1g2)G} implies that $\C_{\FF}((g_1,g_2))G$ is $\cC$-closed in $\FF$.

We have shown that $\C_{\FF}((g_1,g_2))G$ is $\cC$-closed in $\FF$ for every $(g_1,g_2) \in G$, and since $\FF$ is $\cC$-hereditarily conjugacy separable
(by Lemma \ref{lem:C-cs-for_products}), we can use Lemma \ref{lem:crit_for_CS} to conclude that $G$ is $\cC$-conjugacy separable.
\end{proof}

Proposition \ref{prop:crit_of_CS_for_subdirect} can be combined with Lemma \ref{lem:C-open-subdir} to establish $\cC$-hereditary conjugacy separability of subdirect products.

\begin{cor} \label{cor:crit_for_hcs_of_subdir} Let $\cC$ be an extension-closed pseudovariety of finite groups and let $F_1,F_2$ be $\cC$-hereditarily conjugacy separable groups
with cyclic centralizers. If $G \leqslant \FF$ is a subdirect product such that $F_1/N_1$ is highly residually-$\cC$ and
cyclic subgroup $\cC$-separable, where $N_1\coloneq G \cap F_1$, then $G$ is $\cC$-hereditarily conjugacy separable.
\end{cor}

\begin{proof} Consider any $\cC$-open subgroup $H$ of $G$. Then, according to Lemma \ref{lem:C-open-subdir}, there is a $\cC$-open subgroup $J_i \leqslant F_i$, $i=1,2$,
such that  $H \leqslant J_1\times J_2$ is a subdirect product and $J_1/(H\cap J_1)$ is cyclic subgroup $\cC$-separable.

Note that for each $i=1,2$, $J_i$  has cyclic centralizers, as a subgroup of $F_i$, and every $\cC$-open subgroup $K$ of $J_i$ is also $\cC$-open in $F_i$. Indeed, since
the class $\cC$ consists of finite groups, we have $|F_i:J_i|<\infty$ and $|J_i:K|<\infty$, hence $|F_i:K|<\infty$. Now, $K$ is $\cC$-closed in $F_i$ by
Lemma~\ref{lem:restr_crit}.(b) and Remark~\ref{rem:open->closed}, so its complement $F_i \setminus K$ is also $\cC$-closed in $F_i$, being a
finite union of cosets modulo $K$. Therefore $K$ must be $\cC$-open in $F_i$, as the complement of a $\cC$-closed set.

Recalling that $F_i$ is $\cC$-hereditarily conjugacy separable, we can conclude that so is $J_i$, $i=1,2$. It remains to apply Proposition \ref{prop:crit_of_CS_for_subdirect}
to deduce that $H$ is $\cC$-conjugacy separable. Since the latter is true for any $\cC$-open subgroup  $H \leqslant G$, we have shown that $G$
is $\cC$-hereditarily conjugacy separable.
\end{proof}

We will later see why the assumptions that $F_i$ have cyclic centralizers and $F_1/N_1$ is cyclic subgroup $\cC$-separable are essential in
Proposition~\ref{prop:crit_of_CS_for_subdirect} (see Remark~\ref{rem:cyc_central_important} and Subsection \ref{subsec:necessity}).
It is also worth mentioning that some criteria for solvability of the conjugacy problem in subdirect products were studied by Kulikova in \cite{Kul}.

\subsection{Criteria for non-conjugacy separability}
In this subsection $\cC$ will denote a pseudovariety of groups, unless specified otherwise.

We will start with the following general statement.

\begin{lemma} \label{lem:1st_crit_for_non-cs} Let $F_1,F_2$ be groups and let $\rho_1:\FF \to F_1$ denote the natural projection. Assume that
$G \leqslant \FF$ is a subgroup such that $\rho_1(G)=F_1$ and $N_i\coloneq G \cap F_i$, for $i=1,2$.
If $h_i \in N_i$, $i=1,2$, are arbitrary elements and $x_1$ belongs to the closure of $N_1$ in the \proc on $F_1$ then the element $(x_1h_1x_1^{-1},h_2)$ belongs to the
closure of the $N_1$-conjugacy class $(h_1,h_2)^{N_1}\subseteq (h_1,h_2)^G$ in the \proc on $G$.
\end{lemma}

\begin{proof} First, note that $N_1$ is normal in $G$,  and hence it is normal in $F_1$ as $\rho_1(G)=F_1$ (see the proof of Lemma \ref{lem:norm_in_G->norm_in_F}.(i)). Therefore
$x_1h_1x_1^{-1} \in N_1$ and, thus, $(x_1h_1x_1^{-1},h_2) \in N_1\times N_2 \subseteq G$.

Since $\rho_1(G)=F_1$, we can find some element $x_2 \in F_2$ such that $(x_1,x_2) \in G$.
If $K\lhd G$ is any co-$\cC$ subgroup then $\rho_1(K)$ is a co-$\cC$ subgroup of $F_1$, as
$\cC$ is closed under taking quotients, hence $x_1\rho_1(K) \cap N_1 \neq\emptyset$ in $F_1$, because $x_1$ belongs to the closure of $N_1$ in the \proc on $F_1$.
Since $(x_1,x_2) \in \rho_1^{-1}(x_1)$ and $\rho_1^{-1}(N_1)=N_1\times N_2$, it follows that $(x_1,x_2) K \cap (N_1\times N_2) \neq\emptyset$. The latter holds for every
co-$\cC$ subgroup $K$ of $G$, showing that $(x_1,x_2)$ belongs to the closure of $N_1\times N_2$ in the \proc on $G$.

Now, consider any homomorphism $\varphi:G \to M$, where $M \in \cC$. Then $\varphi((x_1,x_2)) \in \varphi(N_1\times N_2)$, i.e., there exist $a_i \in N_i$, $i=1,2$,
such that $\varphi((x_1,x_2))=\varphi((a_1,a_2))$ in $M$. Since the elements
$(h_1,1)$, $(x_1h_1x_1^{-1},1)$,  $(1,h_2)$, $(x_1,x_2)$, $(a_1,a_2)$ and $(a_1,1)$ all belong to $G$, their $\varphi$-images
are defined, and we have
\begin{multline*}
\varphi\left((x_1h_1x_1^{-1},h_2)\right)=\varphi\left((x_1h_1x_1^{-1},1)\right)\varphi((1,h_2))\\=\varphi((x_1,x_2)) \varphi((h_1,1)) \varphi((x_1,x_2))^{-1} \varphi((1,h_2))
=  \varphi((a_1,a_2)) \varphi((h_1,1)) \varphi((a_1,a_2))^{-1} \varphi((1,h_2)) \\
=\varphi\left((a_1h_1a_1^{-1},h_2)\right)=
\varphi\left((a_1,1)(h_1,h_2)(a_1,1)^{-1}\right) \in \varphi\left((h_1,h_2)^{N_1}\right).
 \end{multline*}

Thus we have shown that $\varphi\left((x_1h_1x_1^{-1},h_2)\right) \in \varphi\left((h_1,h_2)^{N_1}\right)$ for every homomorphism $\varphi$ from $G$ to a group $M \in \cC$.
This proves that $(x_1h_1x_1^{-1},h_2)$ belongs to the closure of $(h_1,h_2)^{N_1}$ in the \proc on $G$.
\end{proof}

The following elementary fact will be useful:
\begin{rem} \label{rem:in_cc->double_coset} Suppose that $F$ is any group, $G \leqslant F$ is any subgroup and $f \in F$ is any element. Then,
for an arbitrary  $h \in F$,
$hfh^{-1} \in f^G$ if and only if $h \in G\C_F(f)$.
\end{rem}

We can now formulate the first basic criterion of non-conjugacy separability of subdirect products.
\begin{thm} \label{thm:first_crit} Let $\cC$ be a pseudovariety of groups, let
$G\leqslant \FF$ be a subdirect product of groups $F_1$ and $F_2$, and let $N_i\coloneq G\cap F_i$, $i=1,2$. Suppose that
$F_1/N_1$ is not residually-$\cC$ and there are elements $h_i \in N_i$ such that $\C_{F_i}(h_i) \subseteq N_i$, for $i=1,2$. Then $G$ is not $\cC$-conjugacy separable.
\end{thm}

\begin{proof} The assumption that $F_1/N_1$ is not residually-$\cC$ is equivalent to the statement that $N_1$ is not $\cC$-closed in $F_1$, i.e.,
there is $x_1 \in F_1\setminus N_1$ such that $x_1$ belongs to the closure of $N_1$ in the \proc on $F_1$. Also, note that
\[\C_{\FF}((h_1,h_2))=\C_{F_1}(h_1) \times \C_{F_2}(h_2) \subseteq N_1\times N_2 \subseteq G,\]
therefore $G \C_{\FF}((h_1,h_2))=G$.

In view of Lemma \ref{lem:1st_crit_for_non-cs}, to prove the theorem it is enough to check that $(x_1h_1x_1^{-1},h_2) \notin (h_1,h_2)^G$ in $G$.
Indeed, since  $(x_1h_1x_1^{-1},h_2) =(x_1,1) (h_1,h_2) (x_1,1)^{-1}$ in $\FF$, Remark \ref{rem:in_cc->double_coset} tells us that
this element belongs to $(h_1,h_2)^G$ if and only if $(x_1,1) \in G \C_{\FF}((h_1,h_2))=G$.
But the latter is equivalent to $(x_1,1) \in G \cap F_1 =N_1$, contradicting the
choice of $x_1$.

We can now conclude that $(x_1h_1x_1^{-1},h_2) \notin (h_1,h_2)^G$, but this element belongs to the closure of $(h_1,h_2)^G$ in the \proc on $G$ by
Lemma~\ref{lem:1st_crit_for_non-cs}. It follows that $(h_1,h_2)^G$ is not $\cC$-closed in $G$, thus $G$ is not $\cC$-conjugacy separable by Remark \ref{rem:cs<=>each_cc_closed}.
\end{proof}

Recall that, according to Lemma \ref{lem:inf_norm_sbgp->non-elem}, the existence of elements $h_i \in N_i$ such that $\C_{F_i}(h_i) \subseteq N_i$, for $i=1,2$,
holds as long as $F_i$ are acylindrically hyperbolic groups without non-trivial finite normal subgroups, and $N_i \neq \{1\}$, $i=1,2$. The latter condition simply
means that $G \leqslant \FF$ is a \emph{full} subdirect product.
Theorem \ref{thm:first_crit} together with Lemma \ref{lem:inf_norm_sbgp->non-elem} immediately yield the following:

\begin{cor}\label{cor:non-cs_for_subdir_of_acyl_hyp}
Suppose that $\cC$ is a pseudovariety of groups and $F_i$ is an acylindrically hyperbolic group without non-trivial finite normal subgroups, $i=1,2$.
If $G \leqslant \FF$ is a full subdirect product such that $F_1/N_1$ is not residually-$\cC$ (where $N_1 \coloneq G \cap F_1$) then $G$ is not $\cC$-conjugacy separable.
\end{cor}

Theorem \ref{thm:non-cs_for_subdir_of_hyp} from the Introduction is a special case of Corollary \ref{cor:non-cs_for_subdir_of_acyl_hyp}, since every non-abelian free group or a non-elementary hyperbolic group is acylindrically hyperbolic (see Subsection~\ref{subsec:acyl_hyp}).

We end this subsection by giving an explicit application of Theorem \ref{thm:first_crit} and Lemma \ref{lem:1st_crit_for_non-cs}.

\begin{ex}\label{ex:B-G_gp} Let $P\coloneq \langle a,b \,\|\, bab^{-1}a ba^{-1}b^{-1}=a^2 \rangle$ be the $1$-relator group introduced by Baumslag in \cite{B-gp}.
Baumslag proved that the element $a$ is contained in every subgroup of finite index in $P$, that is, it belongs to the closure of the identity element in
the profinite topology on $P$.

Now, let $F$ be the free group with the free generating set $\{x,y\}$, and let $\psi:F \to P$ be the epimorphism given by $\psi(x)\coloneq a$, $\psi(y) \coloneq b$.
We can construct the symmetric fibre product $G\leqslant F\times F$ corresponding to $\psi$ as in Subsection~\ref{subsec:constr-subdir}.
It can be deduced from the proof of Lemma~\ref{lem:Mih}.(a) that $G=\langle(x,x),(y,y),(h,1) \rangle $, where $h\coloneq yxy^{-1}x yx^{-1}y^{-1}x^{-2}$, because
$N\coloneq \ker \psi$ is the normal closure of $h$ in $F$, by construction.

Note that $x$ belongs to the closure of $N$ in the profinite topology of $F$
(by Lemma~\ref{lem:full_preimage}.(iii)), but $x \notin N$ as $a \neq 1$ in $P$.
Moreover, $\C_F(h)=\langle h \rangle \subseteq N$, as $h$ is not a proper power in the free group $F$.
Therefore the elements $(h,h),(xhx^{-1},h) \in N\times N \leqslant G$ are not conjugate in $G$ (by Remark \ref{rem:in_cc->double_coset}), but are conjugate in every
finite quotient of $G$ by Lemma~\ref{lem:1st_crit_for_non-cs}. In particular, $G$ is not conjugacy separable.
\end{ex}

\subsection{Characterizing conjugacy separable subdirect products of finite index}\label{subsect:C-cs_for_fin_ind}
Corollary~\ref{cor:non-cs_for_subdir_of_acyl_hyp} can be combined with Lemma \ref{lem:C-cs-for_products} to give a complete characterization of
$\cC$-conjugacy separable subdirect products that have finite index

\begin{cor}\label{cor:crit_for_c-cs_if_fin_ind-acyl_case} Let $\cC$ be a non-trivial extension-closed pseudovariety of finite groups, and
let $F_i$ be a $\cC$-hereditarily conjugacy separable acylindrically hyperbolic group
without non-trivial finite normal subgroups, $i=1,2$.
If $G \leqslant \FF$ is a subdirect product of finite index in $\FF$ then the following statements are equivalent:
\begin{itemize}
  \item[(1)] $G$ is $\cC$-conjugacy separable;
  \item[(2)] $F_1/N_1 \in \cC$, where $N_1 \coloneq G \cap F_1$;
  \item[(3)] $G$ is $\cC$-open in $\FF$.
\end{itemize}
\end{cor}

\begin{proof}
Note that $G \cap F_i$ is non-trivial as it has finite index in $F_i$, $i=1,2$. Thus $G \leqslant \FF$ is a full subdirect product.

First let us assume that $G$ is $\cC$-conjugacy separable. Then $F_1/N_1$ is residually-$\cC$ by Corollary~\ref{cor:non-cs_for_subdir_of_acyl_hyp}. But
$|F_1/N_1|=|(\FF):G|<\infty$ (by Lemma \ref{lem:norm_in_G->norm_in_F}.(iii))
and a finite group is residually-$\cC$ if and only if it belongs to $\cC$, thus $F_1/N_1 \in \cC$, and we have shown that
(1) implies (2).

If $F_1/N_1 \in \cC$ then $F_2/N_2 \in \cC$ by Lemma \ref{lem:norm_in_G->norm_in_F}.(ii) (where $N_2\coloneq G \cap F_2$), so
$(\FF)/(N_1\times N_2)\cong F_1/N_1\times F_2/N_2 \in \cC$. Thus $N_1\times N_2$ is a co-$\cC$ subgroup of $\FF$ contained in $G$, so $G$ is $\cC$-open in $\FF$
by Remark \ref{rem:open->closed}. Hence (2) implies (3).

Finally, let us assume (3) and deduce (1). Note that $\FF$ is $\cC$-hereditarily conjugacy separable by Lemma \ref{lem:C-cs-for_products}.
Therefore $G$ is $\cC$-conjugacy separable, as it is $\cC$-open in $\FF$ by the assumption. Thus (3) implies (1).
\end{proof}

Corollary \ref{cor:crit_for_c-cs_if_fin_ind} from the Introduction is a special case of Corollary \ref{cor:crit_for_c-cs_if_fin_ind-acyl_case}
because of Lemma~\ref{lem:free_are_C-hcs}.

\begin{cor}\label{cor:crit_p-cs_for_fin_ind} Suppose that $p$ is a prime, $F_1,F_2$ are non-abelian
free groups and $G\leqslant \FF$ is a subdirect product of finite index. Then the following are equivalent:
\begin{itemize}
  \item[(1)] $G$ is $p$-conjugacy separable;
  \item[(2)] $F_1/N_1$ is a finite $p$-group, where $N_1 \coloneq G \cap F_1$;
  \item[(3)] the index $|(\FF):G|$ is a power of $p$.
\end{itemize}
\end{cor}

\begin{proof} Evidently in this statement $\cC=\cC_p$ is the class of all finite $p$-groups, so the equivalence of (1) and (2) has already been proved in
Corollary \ref{cor:crit_for_c-cs_if_fin_ind}.

The equivalence of (2) and (3) follows from the fact that $|(\FF):G|=|F_1/N_1|$ (see Lemma \ref{lem:norm_in_G->norm_in_F}.(iii)).
\end{proof}

\begin{ex}\label{ex:ind-p}
Let $F$ be the free group of rank $2$. By Lemmas \ref{lem:free_are_C-hcs} and \ref{lem:C-cs-for_products} every finite index subgroup of $F \times F$
is conjugacy separable (with respect to the class of all finite groups), but, in view of Corollary~\ref{cor:crit_p-cs_for_fin_ind},
it is easy to construct finite index subgroups that are not $p$-conjugacy separable for any prime $p$.

\begin{itemize}
  \item Let $p$ be any prime and let $G \leqslant F\times F$
be the symmetric fibre product corresponding to any epimorphism from $F$ to $\Z/p\Z$. Then $G \cap F\times \{1\}=N\times \{1\}$, where $F/N \cong \Z/p\Z$,
and Corollary~\ref{cor:crit_p-cs_for_fin_ind} tells us that $G$ is $p$-conjugacy separable but not $q$-conjugacy separable for any prime $q \neq p$.
  \item We can also take $G \leqslant F\times F$ to be the symmetric fibre product corresponding to any epimorphism from $F$ to $\Z/6\Z$.
In this case, since $\Z/6\Z$ is not a $p$-group, $G$ is not $p$-conjugacy separable for any prime $p$, by Corollary~\ref{cor:crit_p-cs_for_fin_ind}.
\end{itemize}
\end{ex}

More generally, we obtain the following statement.

\begin{cor}\label{cor:dist_C-cs} Let $\cC$ and $\mathcal{D}$ be two non-trivial extension-closed pseudovarieties of finite groups such that
$\cC \not\subseteq \mathcal{D}$, and let $H$ be the
free group of rank $2$. Then there exists a finite index subgroup $G \leqslant H \times H$ such that
$G$ is $\mathcal{D}$-conjugacy separable but not $\cC$-conjugacy separable.
\end{cor}

\begin{proof} By the assumptions, there is some finite group $P \in \mathcal{D} \setminus \cC$. Let $F \leqslant H$ be a finite index subgroup admitting an epimorphism
$\psi:F \to P$.
Then we can construct the symmetric fibre product $G \leqslant F\times F$ corresponding to $\psi$, so that $G \cap (F\times\{1\})=N\times \{1\}$, where $N \coloneq \ker\psi$.

Thus $F/N \cong P \in \mathcal{D}\setminus \cC$, hence $G$ is $\mathcal{D}$-conjugacy separable but not
$\cC$ -conjugacy separable by Corollary \ref{cor:crit_for_c-cs_if_fin_ind}.
Recalling that $|(F\times F):G|=|F/N|<\infty$, by Lemma \ref{lem:norm_in_G->norm_in_F}.(iii), and $|(H\times H):(F\times F)|<\infty$ by construction, we deduce that
$|(H\times H):G|<\infty$, as claimed.
\end{proof}

\subsection{Necessity of cyclic subgroup separability of the quotient}\label{subsec:necessity}
If the reader compares the assumptions on the quotient $F_1/N_1$ in Proposition \ref{prop:crit_of_CS_for_subdirect} and Corollary \ref{cor:non-cs_for_subdir_of_acyl_hyp},
they will immediately notice that to establish $\cC$-conjugacy separability of a subdirect product $G$, Proposition~\ref{prop:crit_of_CS_for_subdirect} requires
$F_1/N_1$ to be cyclic subgroup $\cC$-separable, while Corollary~\ref{cor:non-cs_for_subdir_of_acyl_hyp} only shows that $F_1/N_1$ must be residually-$\cC$ if $G$ is
$\cC$-conjugacy separable. The goal of this subsection is to address this ``gap'': we will give an example showing that it is indeed necessary to require cyclic subgroup $\cC$-separability of $F_1/N_1$, and just residual-$\cC$-ness of $F_1/N_1$ is insufficient.
We will also prove that, in general, if the former condition fails then the corresponding subdirect product possesses a finite
index subgroup which is not $\cC$-conjugacy separable.

The following observation can help in showing that a subgroup of a group is not conjugacy separable (cf. \cite[Remark 3.6]{M-RAAG}); it is a converse of
Lemma~\ref{lem:crit_for_CS}.
\begin{lemma}\label{lem:double_coset-closed}
Let $\cC$ be a pseudovariety of groups, let $F$ be a group, $G \leqslant F$ a subgroup and $f \in F$ an element. If the $G$-conjugacy class $f^G$ is $\cC$-closed
in $F$ then the double coset $\C_F(f)G$ is $\cC$-closed in $F$.
\end{lemma}

\begin{proof}
Suppose that $h \in F \setminus G\C_F(f)$. Then $hfh^{-1} \notin f^G$ by Remark \ref{rem:in_cc->double_coset}, so, since $f^G$ is $\cC$-closed in $F$,
there must exist a co-$\cC$ subgroup $N\lhd F$ such that $\varphi(hfh^{-1}) \notin \varphi(f)^{\varphi(G)}$ in $F/N$, where $\varphi:F \to F/N$ is the natural homomorphism.
It follows that $\varphi(h) \notin \varphi(G) \C_{F/N}(\varphi(f))$ in $F/N$ (see Remark \ref{rem:in_cc->double_coset}),
which obviously yields that $\varphi(h) \notin \varphi\left(G\C_F(f)\right)$ in $F/N$.

Therefore for each  $h \in F \setminus G\C_F(f)$ we found a co-$\cC$ subgroup $N \lhd F$ such that $h \notin  G\C_F(f)N$ (equivalently, $hN \cap G\C_F(f)=\emptyset$),
which shows that the double coset $G\C_F(f)$ is $\cC$-closed in $F$. Since the map $a \mapsto a^{-1}$ is a homeomorphism of $F$, with respect to its
pro-$\cC$ topology, we can conclude that $\C_F(f)G=(G\C_F(f))^{-1}$ is also $\cC$-closed in $F$.
\end{proof}

A priori, it may happen that for a subgroup $G$ of a group $F$, $G$ is $\cC$-conjugacy separable but $g^G$ is not $\cC$-closed in $F$, for some $g \in G$ (even though this conjugacy class is $\cC$-closed in $G$). However, this is certainly impossible if the \proc on $G$ is a restriction of the \proc on $F$.
After combining this observation with Lemma \ref{lem:double_coset-closed} we obtain the following corollary.

\begin{cor}\label{cor:cs+restr->closed_double_cosets} Suppose that $\cC$ is a pseudovariety of groups and $G$ is a $\cC$-conjugacy separable subgroup of a group $F$.
If the \proc on $G$ is a restriction of the \proc on $F$ then  $\C_F(g)G$ is $\cC$-closed in $F$ for each $g \in G$.
\end{cor}

The next statement can be regarded as nearly a converse to Proposition \ref{prop:crit_of_CS_for_subdirect}.

\begin{thm} \label{thm:near_converse}
Let $\cC$ be an extension-closed pseudovariety of finite groups and let $F_1,F_2$ be acylindrically hyperbolic groups without non-trivial finite normal subgroups.
Suppose that $G \leqslant \FF$ is a full subdirect product such that $G$ is $\cC$-conjugacy separable but $F_1/N_1$ is not cyclic subgroup $\cC$-separable, where
$N_1\coloneq G\cap F_1$. Then all of the following must hold:
\begin{itemize}
  \item[(a)] $G$ is $\cC$-closed in $\FF$ but some co-$\cC$ subgroup $G'$ of $G$ is not $\cC$-closed in $\FF$;
  \item[(b)] the quotient $F_1/N_1$ is residually-$\cC$ but not highly residually-$\cC$;
  \item[(c)] $G'$ is not $\cC$-conjugacy separable, thus $G$ is not $\cC$-hereditarily conjugacy separable.
\end{itemize}
\end{thm}

\begin{proof} First, note that for each $i=1,2$, $N_i \coloneq G \cap F_i$ is an infinite normal subgroup of the acylindrically hyperbolic group $F_i$.
Let $h_i \in N_i$ be the loxodromic element provided by Lemma~\ref{lem:inf_norm_sbgp->non-elem},
such that $C_{F_i}(h_i)=\langle h_i \rangle \subseteq N_i$, $i=1,2$.

Since $G$ is $\cC$-conjugacy separable, we can apply Theorem \ref{thm:first_crit} to deduce that $F_1/N_1$ must be residually-$\cC$.
Therefore $G$ is $\cC$-closed in $\FF$ by Corollary \ref{cor:subdir-closed-crit}.
Let us now show that the \proc on $G$ is not a restriction of the \proc on $\FF$.

By the assumptions, $F_1/N_1$ is not cyclic subgroup $\cC$-separable, so
there exists an element $\bar x \in F_1/N_1$ such that $\langle \bar x \rangle$ is not $\cC$-closed in $F_1/N_1$.
Let $\psi:F_1 \to F_1/N_1$ denote the natural epimorphism and let $u \in \psi^{-1}(\bar x)$ be any preimage of $\bar x$ in $F_1$.
Then, by Lemma \ref{lem:inf_norm_sbgp->non-elem},
there exists $m \in \Z$ such that the element $x_1\coloneq h_1^m u\in F_1$ satisfies $\C_{F_1}(x_1)=\langle x_1 \rangle$. Observe that
$\psi(x_1)=\psi(u)=\bar x$ in $F_1/N_1$.

Now, since $G\leqslant \FF$ is subdirect, there exists $v \in F_2$ such that $(x_1,v) \in G$.
After applying Lemma \ref{lem:inf_norm_sbgp->non-elem} once again, we can find some $n \in \Z$ such that the element
$x_2 \coloneq h_2^n v \in F_2$ satisfies $\C_{F_2}(x_2)=\langle x_2 \rangle$. Moreover, since $(1,h_2^n) \in N_2 \subseteq G$, we have that
$(x_1,x_2)=(1,h_2^n)(x_1,v) \in G$.

Observe that in $\FF$ we have \[\C_{\FF}((x_1,x_2))G=\langle (x_1,1),(1,x_2) \rangle G=\langle (x_1,1) \rangle \langle (x_1,x_2) \rangle G=\langle (x_1,1) \rangle G.\]
So, according to Lemma \ref{lem:crit_closed}, to show that the double coset $\C_{\FF}((x_1,x_2))G$ is not $\cC$-closed in $\FF$, it is enough to prove that
$\langle x_1 \rangle N_1$ is not $\cC$-closed in $F_1$. However, $\langle x_1 \rangle N_1=\psi^{-1}(\langle \bar x \rangle) \leqslant F_1$, and
since $\langle \bar x \rangle$ is not $\cC$-closed in $P$, it follows that $\langle x_1 \rangle N_1$ is not closed in the \proc on $F_1$ (see
Lemma \ref{lem:full_preimage}.(iii)). Thus $\C_{\FF}((x_1,x_2))G$ is not $\cC$-closed in $\FF$,  even though
$G$ is $\cC$-conjugacy separable by the assumption. Consequently, we can use Corollary \ref{cor:cs+restr->closed_double_cosets} to conclude that the \proc on $G$ is not
a restriction of the \proc on $\FF$. In view of Lemma \ref{lem:restr_crit}.(a), this means that there is a co-$\cC$ subgroup $G' \lhd G$ such that $G'$ is not closed
in the \proc on $\FF$. Hence $F_1/N_1$ cannot be highly residually-$\cC$ by Lemma \ref{lem:restr_on_subdir}.
Thus we have proved claims (a) and (b), and it remains to prove claim (c).

Let $F_i'$ denote the projection of $G'$ to the $i$-th coordinate group, and let $N_i'\coloneq G'\cap F_i'=G' \cap F_i$, $i=1,2$.
Since $G/G' \in \cC$, $|G:G'|<\infty$ and $F_i'$ is a co-$\cC$ subgroup of $F_i$, $i=1,2$, by our assumptions on $\cC$.
Hence $F_1'\times F_2'$ is a co-$\cC$ subgroup of $\FF$.
It also follows that $|N_i:N_i'|<\infty$, so $N_i'$ must be infinite as $|N_i|=\infty$, $i=1,2$.
Therefore  $G' \leqslant F_1'\times F_2'$ is a full subdirect product, and $G'$ is not $\cC$-closed in
$F_1'\times F_2'$ by Lemma \ref{lem:restr_crit}.(b), because $G'$ is not $\cC$-closed in $\FF$. Therefore $F_1'/N_1'$ is not residually-$\cC$ by
Corollary \ref{cor:subdir-closed-crit}.

Finally, for each $i=1,2$,  $F_i'$ is an infinite normal subgroup of the acylindrically hyperbolic group $F_i$ by construction, so $F_i'$ is itself acylindrically hyperbolic by
\cite[Cor. 1.5]{Osin-acyl} and $\E_{F_i'}(F_i') \subseteq \E_{F_i}(F_i')=\{1\}$ by Lemma \ref{lem:inf_norm_sbgp->non-elem}. Thus $G'$ satisfies all the assumptions of
Corollary~\ref{cor:non-cs_for_subdir_of_acyl_hyp}, which implies that $G'$ is not $\cC$-conjugacy separable, and so claim (c) holds.
\end{proof}

We are now ready to construct an example showing that it is necessary to assume cyclic subgroup $\cC$-separability of $F_1/N_1$ in
Proposition \ref{prop:crit_of_CS_for_subdirect}.

\begin{ex}\label{ex:BS(1,2)}
Let $P=\langle a,t \,\|\, tat^{-1}=a^2\rangle$ be a metabelian Baumslag-Solitar group. Then $P$ is highly residually finite (e.g., by \cite[Cor. 4.17]{Lor}).
It follows that $P$ is highly residually-$\cC$, where $\cC$ is either the class of all finite groups or the class of solvable finite groups.

However, $P$ is not cyclic subgroup $\cC$-separable, because $\langle a \rangle$ is conjugate in $P$ to its proper subgroup $\langle a^2\rangle$, which implies that in any
finite quotient of $P$ these two cyclic subgroups have the same image. Hence $a$ belongs to the closure of $\langle a^2\rangle$ in the \proc on $P$, i.e.,
the cyclic subgroup $\langle a^2 \rangle$ is not $\cC$-closed in $P$.

Now, consider any epimorphism $\psi:F \to P$, where $F$ is the free group of rank $2$, and let $G \leqslant F\times F$ be the resulting symmetric fibre product
(see Subsection \ref{subsec:constr-subdir}). Then $G$ is a full subdirect product in $F\times F$ and $F/(G\cap F) \cong P$ is highly residually-$\cC$.
However, $G$ cannot be $\cC$-conjugacy separable by claim (b) of Theorem \ref{thm:near_converse}.
\end{ex}

Theorem \ref{thm:near_converse} does not quite close the gap between cyclic subgroup $\cC$-separability and  residual-$\cC$-ness of $F_1/N_1$.
Thus the answer to the following natural question is still unknown:

\begin{question}\label{q:css_vs_res-C} Does there exist an extension-closed pseudovariety of finite groups $\cC$ and a full subdirect product
$G \leqslant \FF$, of two non-abelian free groups $F_1,F_2$, such that $G$ is $\cC$-conjugacy separable but $F_1/N_1$ is not cyclic subgroup $\cC$-separable,
where $N_1\coloneq G\cap F_1$?
\end{question}

In Section \ref{sec:p} we will show that the answer to Question \ref{q:css_vs_res-C} is negative when $\cC=\cC_p$ is the class of finite $p$-groups;
we do not have an answer in case when $\cC$ is the class of all finite groups.

\section{An exotic hereditarily conjugacy separable group}\label{sec:non_cs_fi_overgps}
In this section we will see how the criteria from Section \ref{sec:crit} can be combined with the constructions of subdirect products from Subsection~\ref{subsec:constr-subdir}
to produce hereditarily conjugacy separable groups that possess finite index overgroups with  unsolvable conjugacy problem.
We will first give finitely generated (but not finitely presented) examples, as these are easier to construct, before giving finitely presented examples in
Subsection \ref{subsec:fp_ex} (see Theorem \ref{thm:strong_vers}).

Throughout this section  $\cC$ will always be the class of all finite groups,
so we will simply talk about conjugacy separability, cyclic subgroup separability, etc. (suppressing $\cC$).

\subsection{A finitely generated example}
The following lemma will be used to construct finite index overgroups with unsolvable conjugacy problem.

\begin{lemma}\label{lem:constr_of_autom} Let $P$ be a finitely generated group with a finitely generated subgroup $Q \leqslant P$. Then for every integer $k\ge 2$ there exists
a finitely generated free group $F$, an automorphism $\sigma \in Aut(F)$ and an epimorphism $\psi:F \to P$ such that
\begin{itemize}
\item the order of $\sigma$ is $k$;
  \item $\psi(\Fix(\sigma))=Q$, where $\Fix(\sigma)  \coloneq \{h \in F \mid \sigma(h)=h\}$ is the subgroup of~fixed~points~of~$\sigma$;
  \item $\sigma$ induces the identity automorphism of $P$, i.e., $\sigma(\ker \psi)=\ker\psi$ and $\psi(\sigma (f))=\psi(f)$ for all $f \in F$.
\end{itemize}
\end{lemma}

\begin{proof} Clearly we can suppose that $P$ is generated by some elements $a_1,\dots, a_m, b_1,\dots,b_n \in P$, such that $a_1,\dots,a_m$ generate $Q$.
Take $F=F(Z)$ to be the free group on a set \[Z\coloneq \{x_1,\dots,x_m,y_{11},\dots,y_{1k}, \dots, y_{n1},\dots,y_{nk} \}\] of cardinality $m+kn$.
Let $\psi:F \to P$ be the epimorphism defined by $\psi(x_l)\coloneq a_l$ and $\psi(y_{ij})\coloneq b_i$, for $l=1,\dots, m$, $i=1,\dots,n$ and $j=1,\dots,k$.

Now let $\sigma:F \to F$ be the automorphism given by the following permutation of $X$:
\[\sigma(x_l)\coloneq x_l, \mbox{ for all } l=1,\dots,m,~\sigma(y_{ij})\coloneq y_{i,j+1}, \mbox{ for all } i=1,\dots,n, j=1,\dots,k,\]
where the addition of indices is done modulo $k$. Evidently $\sigma$ has order $k$ in $Aut(F)$ and induces the identity automorphism of $P$.

Clearly,  $\langle x_1,\dots,x_m \rangle \subseteq \Fix(\sigma)$. On the other hand, if $w \in F$ is a reduced word fixed by $\sigma$, then $\sigma(w)$ is also a reduced word
of the same length. So $\sigma(w)=w$ can only occur if $w$ consists entirely of letters from $\{x_1,\dots,x_m\}^{\pm 1}$, which shows that
$\Fix(\sigma) \subseteq \langle x_1,\dots,x_m \rangle$. Hence $\Fix(\sigma) = \langle x_1,\dots,x_m \rangle$, and so
$\psi(\Fix(\sigma))=\langle \psi(x_1),\dots\psi(x_m) \rangle = Q$, as required.
\end{proof}

Let $F$ be a group generated by a finite set $Z$ and let $Y \subseteq F$ be a subset. We will say that the \emph{membership problem for $Y$ in $F$ is solvable}
if there is an algorithm which, given any word $W$ over $Z^{\pm 1}$, decides whether or not $W$ represents an element of $Y$ in $F$.
The \emph{conjugacy problem in $F$ is solvable} if there exists an algorithm taking on input two words over $Z^{\pm 1}$, and deciding whether or not the elements of $F$
represented by these words are conjugate in $F$.

By a \emph{normal overgroup} of a group $G$ we mean a group $K$, which contains a normal subgroup isomorphic to $G$
(to simplify the notation we will identify $G$ with this normal subgroup).

\begin{thm}\label{thm:weaker_vers} For every integer $k \ge 2$ there exists a finitely generated subdirect product
$G \leqslant F\times F$, where $F$ is a finitely generated non-abelian free group, satisfying the following. The group
$G$ is hereditarily conjugacy separable but there is a normal overgroup $K$, of $G$, such that $|K:G|=k$, $K$ is not conjugacy separable and has unsolvable conjugacy problem.
\end{thm}

\begin{proof} Let $P$ be a finitely presented group satisfying the following three conditions:
\begin{itemize}
  \item[1.] $P$ is highly residually finite;
  \item[2.] $P$ is cyclic subgroup separable;
  \item[3.] there is a finitely generated subgroup $Q\leqslant P$ such that the membership problem for $Q$ in $P$ is unsolvable.
\end{itemize}
For example, we can take $P$ to be the direct product of two free groups of rank $2$. Indeed, such a group $P$ is highly residually finite because for finitely generated groups
this property is stable under direct products (cf. \cite[Cor. 2.11]{Lor}) and free groups obviously have it (as any finite-by-free group is virtually free). That
direct products of free groups are cyclic subgroup separable can be extracted from \cite[Thm. 4.4]{Bur-Mar}. Finally, the existence of a finitely generated subgroup $Q$
with unsolvable membership problem in the direct product of two free groups of rank $2$ was proved by  Miha{\u\i}lova in \cite[Thm. 1]{Mih}.

Let $F$ be the finitely generated free group, $\sigma \in Aut(F)$ be the automorphism and $\psi:F \to P$ be the epimorphism given by Lemma \ref{lem:constr_of_autom}.
We also let $G\leqslant F\times F$ be the symmetric fibre product corresponding to $\psi$. Then $G$ is finitely generated
by Lemma~\ref{lem:Mih}.(a), and $G$ is hereditarily conjugacy separable by Corollary \ref{cor:crit_for_hcs_of_subdir} and Lemma \ref{lem:free_are_C-hcs}.

Now, let $\tilde F \coloneq F \rtimes_\sigma \langle t \rangle_k$ be the semidirect product of the free group $F$ with the cyclic group ${\langle t \rangle}_k$ of order $k$,
where $tft^{-1}\coloneq \sigma(f)$ for all $f \in F$. Since $\sigma(N)=N$, where $N \coloneq\ker\psi$, and $\sigma$ induces the identity automorphism of $P$,
$\psi$ can be naturally extended to an epimorphism $\tilde \psi: \tilde F \to P \times \langle t \rangle_k$, by defining $\tilde\psi(f)=\psi(f)$, for all $f \in F$,
and $\tilde\psi(t)=t$. Then $\ker\tilde\psi=N$, so if $K\leqslant \tilde F \times \tilde F$ is the symmetric fibre product corresponding to $\tilde \psi$ then
$K \cap (\tilde F \times \{1\})=N\times \{1\}$. Evidently we can identify $G$ with the subgroup of $K$ defined by
\[\{(g_1,g_2) \in F\times F \leqslant \tilde F\times \tilde F \mid \tilde \psi(g_1)=\tilde\psi(g_2)\} \leqslant K.\]

In other words, $G=K \cap (F \times F)$ in $\tilde F \times \tilde F$. By construction, $F$ is normal in $\tilde F$, yielding that $G$ is normal in $K$.
Recall that $K/(N\times \{1\}) \cong \tilde F$ and $G/(N\times \{1\})\cong F$, so
\[K/G\cong \frac{K/(N\times \{1\})}{G/(N\times \{1\})} \cong \tilde F/F \cong \langle t \rangle_k, \mbox{ thus } |K:G|=k.\]

It remains to show that $K$ is not conjugacy separable and has unsolvable conjugacy problem. Observe that for every $x \in F$, $\tilde\psi(x^{-1}tx)=\tilde\psi(t)$, because
$t=\tilde\psi(t)$ is central in $P\times \langle t \rangle_k$. Therefore $(x^{-1}tx,t) \in K$ for each $x \in F$.

Now, for any $x \in F$, since $(x^{-1}tx,t)=(x^{-1},1)(t,t)(x^{-1},1)^{-1}$,  the element $(x^{-1}tx,t)$ is conjugate to $(t,t)$ in $K$
if and only if $(x^{-1},1) \in K \C_{\tilde F\times \tilde F}((t,t))$ (see Remark \ref{rem:in_cc->double_coset})
if and only if $(x,1) \in  \C_{\tilde F\times \tilde F}((t,t))K$ if and only if $(x,1) \in  (\C_{\tilde F}(t) \times \C_{\tilde F}(t))K=(\C_{\tilde F}(t) \times \{1\})K$
(because $K$ contains the diagonal subgroup of $\tilde F\times \tilde F$, by definition) if and only if $x \in \C_{\tilde F}(t) N$ (cf. Remark~\ref{rem:YG_int_F_1})
if and only if
$x \in \C_F(t) N$ (as $x \in F$ and $N \subseteq F$, so $F \cap \C_{\tilde F}(t)N= \C_F(t) N$).

Thus, if we were able to solve the conjugacy problem in $K$, then we would be able to solve the membership problem for  $\C_F(t) N$ in $F$. But
$\C_F(t)=\Fix(\sigma)$ and $\psi(\Fix(\sigma))=Q$, hence $\psi^{-1}(Q)=\C_F(t)N$. And so the membership problem for $\C_F(t) N$ in $F$ is equivalent to the
membership problem for $Q$ in $P$, which is undecidable by construction. Therefore the conjugacy problem in $K$ is unsolvable.

Finally, let us show that $K$ is not conjugacy separable. Note that we cannot use Mostowski's result  \cite[Thm. 3]{Mostowski}
for this, as $K$ is not finitely presented by a theorem of Baumslag and Roseblade \cite[Thm. B]{Baum-Rose}.
Since $P$ is a finitely presented group and the membership problem for the finitely generated subgroup $Q$ in $P$ is unsolvable,
$Q$ cannot be closed in the profinite topology of $P$ by a standard Mal'cev-type argument (see \cite[\S 7]{Malcev} or \cite[Thm. 2]{Mostowski}). Therefore $\C_F(t) N=\psi^{-1}(Q)$ is not closed in the profinite topology of $F$ by
Lemma \ref{lem:full_preimage}.(iii).
Observe that $ F \times \{1\} \cap \C_{\tilde F\times \tilde F}((t,t))K=\C_F(t) N \times \{1\}$ (we have essentially shown this earlier in the proof),
so the double coset $\C_{\tilde F\times \tilde F}((t,t))K$
is not closed in the profinite topology of $\tilde F \times \tilde F$ by Remark~\ref{rem:pro-C_on_sbgp}. Therefore $(t,t)^K$ is not closed in the profinite topology
of $\tilde F\times \tilde F$ by Lemma~\ref{lem:double_coset-closed}. But the profinite topology on $K$ is the restriction of the profinite topology on
$\tilde F \times \tilde F$ by Lemma~\ref{lem:restr_on_subdir}, because $\tilde F /N \cong P \times \Z/k\Z$ is highly residually finite (e.g., by \cite[Cor. 2.11]{Lor}).
Consequently, we can conclude that the conjugacy class $(t,t)^K$ is not closed in the profinite topology of $K$, thus $K$ is not conjugacy separable and the proof is complete.
\end{proof}

\begin{rem}\label{rem:cyc_central_important} Theorem \ref{thm:weaker_vers} shows that it is necessary to assume in Proposition \ref{prop:crit_of_CS_for_subdirect}
that the groups $F_i$ have cyclic centralizers, $i=1,2$. Indeed, the group $\tilde F$, the subdirect product $K \leqslant \tilde F\times \tilde F$ and the normal subgroup
$N\lhd \tilde F$ which were constructed in the proof of Theorem~\ref{thm:weaker_vers} clearly satisfy all the remaining assumptions of
Proposition \ref{prop:crit_of_CS_for_subdirect}, nevertheless, $K$ is not conjugacy separable.
\end{rem}

\subsection{Finitely presented examples} \label{subsec:fp_ex}
To prove Theorem \ref{thm:strong_vers} we will need an enhancement of Lemma \ref{lem:constr_of_autom}. It will use classical small cancellation theory,
and we refer the reader  to \cite[Sec. V.2]{L-S} for the background and definitions. Some of the basic properties of small cancellation groups are summarized in the next statement.

 \begin{lemma}\label{lem:small_canc->hyp}
Suppose that $\lambda \in (0,1/6]$ and $F$ is a group given by  a finite presentation $\langle Z\,\|\, \mathcal{R} \rangle$ satisfying
the small cancellation condition $C'(\lambda)$. Then
\begin{itemize}
  \item[(i)] $F$ is hyperbolic;
  \item[(ii)] $F$ has cyclic centralizers;
  \item[(iii)] if $f \in F$ is an element of finite order then there is a word $R$ over $Z^{\pm 1}$ and an integer $n \ge 2$ such that
  $R^n \in \mathcal R$ is a defining relator and $f$ is conjugate in $F$ to an  element represented by a power of $R$.
\end{itemize}
\end{lemma}

\begin{proof} (i) The group $F$ is hyperbolic because, by Greendlinger's lemma \cite[Thm. 4.5 in Sec. V.4]{L-S}, any $C'(1/6)$ presentation
is necessarily a \emph{Dehn presentation} and any group with a finite Dehn presentation is hyperbolic (cf. \cite[Thm. 2.6 in Ch. III.$\Gamma$]{B-H}).

(ii) That centralizers of non-trivial elements in a $C'(1/6)$-group are cyclic was proved by Truffault \cite{Truf}.

The statement (iii), in this generality, is due to Greendlinger \cite[Thm. VIII]{Greend-torsion}.
\end{proof}

The main technical tool in the proof of Theorem \ref{thm:strong_vers} is the following proposition, which may also have other applications in the future.

\begin{prop}\label{prop:constr_of_autom_hyp}
Let $P$ be a finitely presented group and let $Q \leqslant P$ be a finitely generated subgroup. Then for each real number $\lambda>0$ and every integer $k\ge 2$ there exists
a torsion-free group $F$, with a finite $C'(\lambda)$ small cancellation presentation, an automorphism $\sigma \in Aut(F)$ and an epimorphism $\psi:F \to P$
such that
\begin{itemize}
\item $\ker \psi$ is generated by $k$ elements;
\item the order of $\sigma$ is $k$;
  \item $\psi(\Fix(\sigma))=Q$, where $\Fix(\sigma)  \coloneq \{h \in F \mid \sigma(h)=h\}$ is the subgroup of fixed~points~of~$\sigma$;
  \item $\sigma$ induces the identity automorphism of $P$, i.e., $\sigma(\ker \psi)=\ker\psi$ and $\psi(\sigma (f))=\psi(f)$ for all $f \in F$.
\end{itemize}
\end{prop}

\begin{proof} The argument is essentially a combination of Rips's original idea from \cite[Thm.]{Rips} with the proof of Lemma \ref{lem:constr_of_autom}.

As before, we can suppose that $P$ is generated by some elements $a_1,\dots, a_m, b_1,\dots,b_n$ and
$Q=\langle a_1,\dots,a_m\rangle$. Let $\langle a_1,\dots,a_m,b_1,\dots,b_n \,\|\, R_1,\dots,R_s\rangle$ be a finite  presentation of $P$.

Let $F$ be the group given by a finite presentation $\langle Z\,\|\,\mathcal{R}\rangle$, where
\[Z \coloneq \{x_1,\dots,x_m, y_{11},\dots,y_{1k},\dots,y_{n1},\dots,y_{nk},z_1,\dots,z_k \}\]
is a generating set of cardinality $m+nk+k$. To specify the set of defining relators $\mathcal R$, for all $t \in \{1,\dots,s\}$ and $j \in \{1,\dots,k\}$,
let $R_{tj}$ be the word over the alphabet $\{x_1,\dots,x_m,y_{1j}\dots,y_{nj}\}^{\pm 1}$ obtained from the word $R_t$ by
replacing each letter $a_l$ with $x_l$ and each $b_i$ with $y_{ij}$, $l=1,\dots, m$, $i=1,\dots,n$.

Let
\begin{equation}\label{eq:integers}
\alpha_t<\beta_t,~ \gamma_i<\delta_i,~ \e_l<\zeta_l,~ \e'_l<\zeta'_l,~ \eta_{ip}<\theta_{ip} \mbox{ and }\eta'_{ip}<\theta'_{ip},
\end{equation}
$t=1,\dots,s$, $i=1,\dots,n$, $l=1,\dots,m$, $p=0,\dots,k-1$, be some large positive integers that we will specify later.
The set of defining relators $\mathcal{R}$, of $F$, will consist of the following words (where the addition of lower indices at $z_\ast$ and $y_{i,\ast}$ is done modulo $k$):
\begin{equation}\label{eq:1-rel}
R_{tj} z_j z_{j+1}^{\alpha_{t}}z_jz_{j+1}^{\alpha_{t}+1} z_jz_{j+1}^{\alpha_{t}+2}\dots z_jz_{j+1}^{\beta_{t}}, ~~t=1,\dots,s,~j=1,\dots,k,
\end{equation}

\begin{equation}\label{eq:2-rel}
y_{ij}^{-1} y_{i,j+1} z_j z_{j+1}^{\gamma_{i}}z_jz_{j+1}^{\gamma_{i}+1} z_jz_{j+1}^{\gamma_{i}+2}\dots z_jz_{j+1}^{\delta_{i}}, ~~ i=1,\dots,n,~j=1,\dots,k,
\end{equation}

\begin{equation}\label{eq:3-rel}
x_l^{-1} z_j x_l z_j z_{j+1}^{\e_{l}}z_jz_{j+1}^{\e_{l}+1} z_jz_{j+1}^{\e_{l}+2}\dots z_jz_{j+1}^{\zeta_{l}}, ~~ l=1,\dots,m,~j=1,\dots,k,
\end{equation}

\begin{equation}\label{eq:4-rel}
x_l z_j x_l^{-1} z_j z_{j+1}^{\e'_{l}}z_jz_{j+1}^{\e'_{l}+1} z_jz_{j+1}^{\e'_{l}+2}\dots z_jz_{j+1}^{\zeta'_{l}}, ~~ l=1,\dots,m,~j=1,\dots,k,
\end{equation}

\begin{equation}\label{eq:5-rel}
y_{i,p+j}^{-1} z_j y_{i,p+j} z_j z_{j+1}^{\eta_{ip}}z_jz_{j+1}^{\eta_{ip}+1} \dots z_jz_{j+1}^{\theta_{ip}}, ~~ i=1,\dots,n,~p=0,\dots,k-1,~j=1,\dots,k,
\end{equation}

\begin{equation}\label{eq:6-rel}
y_{i,p+j} z_j y_{i,p+j}^{-1} z_j z_{j+1}^{\eta'_{ip}}z_jz_{j+1}^{\eta'_{ip}+1} \dots z_jz_{j+1}^{\theta'_{ip}},
~~ i=1,\dots,n,~p=0,\dots,k-1,~j=1,\dots,k.
\end{equation}

Clearly for $\mu \coloneq \min\{\lambda,1/8\}$ we can choose the positive integers \eqref{eq:integers} in such a way that all the intervals
$[\alpha_t,\beta_t]$, $[\gamma_i,\delta_i]$, $[\e_l,\zeta_l]$,
$[\e'_l,\zeta'_l]$, $[\eta_{ip},\theta_{ip}]$  and $[\eta'_{ip},\theta'_{ip}]$,
$t=1,\dots,s$, $i=1,\dots,n$, $l=1,\dots,m$, $p=0,\dots,k-1$, are very long (compared to the maximum of the lengths of the words $R_{tj}$) and pairwise disjoint, so that
the set $\mathcal R$ satisfies the small cancellation condition $C'(\mu)$. It follows that $\mathcal{R}$ satisfies both $C'(\lambda)$ and $C'(1/8)$. Note that no
defining relator from $\mathcal{R}$ is a proper power, hence $F$ is torsion-free by Lemma \ref{lem:small_canc->hyp}.(iii).

Since the indices $l$, $i$ and $p$ are independent of $j$ in \eqref{eq:3-rel}--\eqref{eq:6-rel}, these relators ensure that $N\coloneq \langle z_1,\dots,z_k\rangle$
is normal in $F$, and the relators \eqref{eq:1-rel}--\eqref{eq:2-rel} ensure that the quotient $F/N$ is naturally isomorphic to $P$. More precisely, we can define the
epimorphism $\psi:F \to P$, with $\ker\psi=N$, by setting $\psi(x_l) \coloneq a_l$, $\psi(y_{ij})\coloneq b_i$ and $\psi(z_j)=1$, for $l=1,\dots,m$,
$i=1,\dots,n$ and $j=1,\dots,k$.

Now, define a permutation $\sigma$ of the generating set $Z$ of $F$ by
\[\sigma(x_l)\coloneq x_l,~\sigma(y_{ij})\coloneq y_{i,j+1}, ~\sigma(z_j)\coloneq z_{j+1}, \]  for all $l=1,\dots,m$, $i=1,\dots,n$, $j=1,\dots,k$.
Naturally, $\sigma$ extends to a permutation of the set of words over $Z^{\pm 1}$, so that each group of the defining relators \eqref{eq:1-rel}--\eqref{eq:6-rel}
is invariant under $\sigma$ by construction (one can see that $\sigma$ acts by adding $1$ to the index $j$ modulo $k$).
Therefore $\sigma(\mathcal{R})=\mathcal{R}$, hence $\sigma$ defines an automorphism of $F$, which, by abusing the notation, we will again denote by
$\sigma \in Aut(F)$. Clearly $\sigma^k$ is the identity automorphism of $F$, but $\sigma^\varkappa$ is non-trivial in $Aut(F)$, if $1 \le \varkappa < k$, because
$z_1 \neq z_{1+\varkappa}=\sigma^\varkappa(z_1)$ in $F$ (otherwise, by Greendlinger's lemma \cite[Thm. 4.5 in Sec. V.4]{L-S}, the cyclically reduced
word $z_1^{-1}z_{1+\varkappa}$, of length $2$, would contain more than a half of a cyclic permutation of a word from $\mathcal{R}^{\pm 1}$,
but the length of any word from $\mathcal{R}$ is greater than $8$,
by construction). Therefore the order of $\sigma$ in $Aut(F)$  is $k$. Obviously $\sigma$ induces the identity automorphism of $P$.

It remains to show that $\Fix(\sigma)=\langle x_1,\dots,x_m\rangle$ in $F$. Evidently, $\langle x_1,\dots,x_m\rangle \subseteq \Fix(\sigma)$ by the definition of $\sigma$.
Arguing by contradiction, assume that the converse inclusion does not hold, and take an element
$w \in \Fix(\sigma)\setminus\langle x_1,\dots,x_m\rangle$ such that $w$ has the  shortest possible
length when expressed as a word over the alphabet $Z^{\pm 1}$. Choose a geodesic (i.e., of minimal length) word $W$ over this alphabet representing $w$ in $F$.

Since $\sigma(w)=w$, the word $W^{-1}\sigma(W)$ represents the identity element of $F$. Obviously, the words $W^{-1}$ and
$\sigma(W)$ are both freely reduced and geodesic in $F$, as $W$ is geodesic.
Also, note that $W$ cannot start with $x_l^{\xi}$, where $\xi =\pm 1$ and $l \in \{1,\dots,m\}$,
because otherwise $x_l^{-\xi}w$ would be an element of $\Fix(\sigma)\setminus\langle x_1,\dots,x_m\rangle$ which is strictly shorter than $w$.
Similarly, $W$ cannot end with a
letter from $\{x_1,\dots,x_m\}^{\pm 1}$. It follows that the first letter of $\sigma(W)$ is not the inverse of the last letter of $W^{-1}$, and the last letter
of $\sigma(W)$ is not the inverse of the first letter of $W^{-1}$. Therefore the word $W^{-1}\sigma(W)$ is freely cyclically reduced,
and since $\mathcal{R}$ satisfies $C'(\mu)$,
we can apply \cite[Thm. 4.4 in Sec. V.4]{L-S}. This theorem claims that there is a cyclic permutation $S$, of a word from $\mathcal{R}^{\pm 1}$, and a prefix $U$,
of $S$, such that $U$ is  a subword of $W^{-1}\sigma(W)$ and $\|U\|>(1-3\mu) \|S\|$, where $\|U\|$ denotes the length of the word $U$.

Now, observe that $1-3\mu>1/2$ as $\mu \le 1/8$, so the word $U$ contains more than a half of the relator $S$, of $F$, and hence it is not geodesic in $F$. Therefore
$U$ cannot be solely a subword of $W^{-1}$ or of $\sigma(W)$, each which is geodesic by construction. Consequently, we can write $U\equiv U_1U_2$, where $U_1$ is the
suffix of $W^{-1}$ and $U_2$ is a prefix of $\sigma(W)$, such that $\max\{\|U_1\|,\|U_2\|\} \le \frac12 \|S\|$. Since $\|U_1\|+\|U_2\|=\|U\|>(1-3\mu)\|S\|$, we deduce
that $\min\{\|U_1\|,\|U_2\|\} > (1/2-3\mu)\|S\|$.

We will now assume that $\|U_1\|\le \|U_2\|$, as the case when $\|U_2\|<\|U_1\|$ can be treated similarly.
Let us replace $U_2$ with its prefix that has the same length as $U_1$. Then $U_1$, $U_2$ are disjoint subwords of $S$ satisfying
$\|U_1\|=\|U_2\|>(1/2-3\mu)\|S\|$. Moreover, since $U_1$ is the suffix of $W^{-1}$ and $U_2$ is the prefix of $\sigma(W)$,  the word
$\sigma(U_1)^{-1}$ coincides with the word $U_2$. But $\sigma(U_1)^{-1}$ is a suffix of the word $\sigma(S)^{-1}$, which is a cyclic permutation of some defining relator from
$\mathcal{R}^{\pm 1}$ because $\sigma(\mathcal{R})=\mathcal{R}$. Thus $U_2 \equiv \sigma(U_1)^{-1}$ is a common subword of two cyclic permutations $S$ and $\sigma(S)^{-1}$,
of words from $\mathcal{R}^{\pm 1}$.
Note that by the construction of the defining relators \eqref{eq:1-rel}--\eqref{eq:6-rel}, no
relator $T$ is equal to a cyclic permutation of $\sigma(T)^{-1}$. Therefore the word $U_2 $ is a piece of $S$
(in the terminology of \cite[Sec. V.2]{L-S}), whose length is greater than $(1/2-3\mu)\|S\|\ge \frac18 \|S\|$, as $\mu \le 1/8$, contradicting the fact that the
set $\mathcal{R}$ satisfies $C'(1/8)$.

Therefore there can be no elements $w \in \Fix(\sigma) \setminus \langle x_1,\dots,x_m\rangle$.
Thus $\Fix(\sigma) = \langle x_1,\dots,x_m\rangle$, so
\[\psi(\Fix(\sigma))=\langle\psi(x_1)\dots,\psi(x_m)\rangle=\langle a_1,\dots,a_m \rangle =Q.\]
This finishes the proof of the proposition.
\end{proof}

\begin{thm}\label{thm:strong_vers} For each real number $\lambda>0$ and every integer $k \ge 2$ there exists a finitely presented subdirect product
$G \leqslant F\times F$, where $F$ is a finitely presented torsion-free $C'(\lambda)$-group, satisfying the following.
The group $G$ is hereditarily conjugacy separable but
there is a normal overgroup $K$, of $G$, such that $|K:G|=k$, $K$ is not conjugacy separable and has unsolvable conjugacy problem.
\end{thm}

\begin{proof} Take any finitely presented group $P$ satisfying the following four conditions:
\begin{itemize}
  \item[1.] $P$ is highly residually finite;
  \item[2.] $P$ is cyclic subgroup separable;
  \item[3.] there is a finitely generated subgroup $Q\leqslant P$ such that the membership problem for $Q$ in $P$ is unsolvable;
  \item[4.] $P$ is of type $\mathrm{F}_3$.
\end{itemize}
As before, we can take $P$ to be the direct product of two free groups of rank $2$ (in the proof of Theorem \ref{thm:weaker_vers} we have already explained that
it would satisfy conditions 1--3, and, obviously, it would also satisfy condition 4).

Take $\mu\coloneq \min \{\lambda,1/6\}$ and apply Proposition \ref{prop:constr_of_autom_hyp} to find a torsion-free $C'(\mu)$ group $F$, an automorphism $\sigma \in Aut(F)$ and
an epimorphism $\psi:F \to P$ from its claim. Then $F$ has cyclic centralizers by Lemma \ref{lem:small_canc->hyp}.(ii).

Let $G\leqslant F\times F$ be the symmetric fibre product corresponding to $\psi$. Then $G$ is finitely presented by
Lemma \ref{lem:1-2-3}, because $N\coloneq \ker\psi$ is finitely generated, $F$ is finitely presented and $P$ is of type $\mathrm{F}_3$.
Observe that $F$ is hereditarily conjugacy separable by \cite[Cor. 1.3]{M-Z-vcs}, and, hence, so is $G$ by Corollary \ref{cor:crit_for_hcs_of_subdir}.

As before, we let $\tilde F \coloneq F \rtimes_\sigma \langle t \rangle_k$ be the semidirect product of the group $F$ with the cyclic group
${\langle t \rangle}_k$ of order $k$, where $tft^{-1}\coloneq \sigma(f)$ for all $f \in F$. By construction,
$\psi$ extends to an epimorphism $\tilde \psi: \tilde F \to P \times \langle t \rangle_k$, where $\tilde\psi(f)=\psi(f)$, for all $f \in F$,
$\tilde\psi(t)=t$ and $\ker\tilde\psi=N$. Finally, we let $K\leqslant \tilde F \times \tilde F$ be the symmetric fibre product corresponding to $\tilde \psi$.

It remains to repeat the arguments from the proof of Theorem \ref{thm:weaker_vers} to show that $G$ can be identified with a normal subgroup of index $k$ in $K$,
$K$ has unsolvable conjugacy problem and $K$ is not conjugacy separable. (Observe that in this case the fact that $K$ is not conjugacy separable can be deduced from
Mostowski's result \cite[Thm. 3]{Mostowski}, because $K$ is finitely presented and has unsolvable conjugacy problem.)
\end{proof}

\section{Conjugacy separability of finite index overgroups} \label{sec:cs_overgroups}
In this section we will first show that a group which is not hereditarily conjugacy separable always has a finite index overgroup which is not conjugacy separable.
Our second goal will be to produce an example of a finitely presented group $G$ possessing an non-conjugacy separable subgroup of index $2$, such that every
finite index normal overgroup of $G$ is conjugacy separable. This section is still concerned with the case when $\cC$ is the class of all finite groups.

\subsection{Constructing non-conjugacy separable overgroups from subgroups}\label{subsec:contr_of_non_cs_overgp_from_sbgp}
Given a group $G$, a subgroup $H \leqslant G$ and elements $x,y \in G$, we will write $x \sim_H y$ if there exists $h \in H$ such that $x=hyh^{-1}$.
If no such $h \in H$ exists, then we will write $x \not\sim_H y$.

Let us start with the following easy observation.

\begin{lemma}\label{lem:ind_2_sbgp->ind_3_overgp} Let $G$ be a group with a subgroup $H \leqslant G$ of index $2$. Then there are an overgroup $K$, of $G$,
with $|K:G|=3$, and an element $a \in K$, centralizing $H$, such that for any $x \in H$ and $y \in G$, $x \sim_H y$ if and only if $ax \sim_K ay$ in $K$.
\end{lemma}

\begin{proof} Let $K \coloneq \langle a \rangle_3 \rtimes G$, where $G$ acts on the cyclic group $\langle a \rangle_3$, of order $3$, as follows:
\[gag^{-1}=\left\{
\begin{array}{cc}
  a & \mbox{if } g \in H \\
  a^2 & \mbox{if } g \notin H
\end{array}
\right. .\] This action is well-defined since $|G/H|=2$ and $a \mapsto a^2$ is an automorphism of $\langle a \rangle_3$ of order $2$.

Clearly $|K:G|=3$ and $a$ centralizes $H$, so it remains to check the last claim. Let $x \in H$ and $y \in G$ be arbitrary elements.
If there is $h \in H$ such that $x=hyh^{-1}$ then
$h(ay)h^{-1}=ahyh^{-1}=ax$, because $h$ commutes with $a$, hence $ax$ is conjugate to $ay$ in $K$.

Conversely, suppose that $ax \sim_K ay$. Then there exist $g \in G$ and $\e \in \{0,1,2\}$ such that $a^\e g(ay)g^{-1}a^{-\e}=ax$ in $K$.
This  is equivalent to $g(ay)g^{-1}=ax$, as $a^\e$ commutes with $ax \in aH$, thus we have $(gag^{-1}) (gyg^{-1})=ax$.
Since $gag^{-1},a \in \langle a \rangle$, $gyg^{-1},x \in G$ and $K$ is the semidirect product of $\langle a \rangle$ with $G$, we can deduce that $gag^{-1}=a$
and $gyg^{-1}=x$. But the former equality implies that $g \in H$, so the latter equality yields $x \sim_H y$.

Thus we have proved that $x \sim_H y$ is equivalent to $ax \sim_K ay$, as required.
\end{proof}

Lemma \ref{lem:ind_2_sbgp->ind_3_overgp} immediately gives the following corollary.

\begin{cor}\label{cor:ind_2->overind_3_CP} If $G$ is a finitely generated group possessing a subgroup $H$, of index $2$, which has unsolvable conjugacy problem,
then $G$ has an overgroup $K$, with $|K:G|=3$, such that $K$ has unsolvable conjugacy problem.
\end{cor}

We can also deduce the analogous fact for conjugacy separability.

\begin{cor}\label{cor:ind_2->overind_3} If $G$ is a group possessing a subgroup $H$, of index $2$, which is not conjugacy separable,
then $G$ has an overgroup $K$, with $|K:G|=3$, such that $K$ is not conjugacy separable.
\end{cor}

\begin{proof} Let $K$ be the overgroup of $G$ and let $a \in K$ be the element centralizing $H$, given by Lemma \ref{lem:ind_2_sbgp->ind_3_overgp}.
Since $H$ is not conjugacy separable
by the assumptions, there are two elements $x,y \in H$ such that $x \not\sim_H y$ but $x$ is conjugate to $y$ in every finite quotient of $H$.

Then $ax \not\sim_K ay$, but for every homomorphism $\varphi:K \to M$, where $M$ is a finite group, we have $\varphi(x) \sim_{\varphi(H)} \varphi(y)$, which implies that
$\varphi(ax) \sim_M \varphi(ay)$, as $\varphi(a)$ commutes with every element of $\varphi(H)$ in $M$. Therefore $K$ is not conjugacy separable.
\end{proof}

Theorem \ref{thm:norm_overgps-cs} below shows that the index $|K:G|=3$ is optimal in Corollary \ref{cor:ind_2->overind_3}, as an index $2$ overgroup would necessarily be normal.
The next proposition deals with the general case. It gives an exponential bound on the index $|K:G|$ in terms of $|G:H|$, which is not always optimal.

\begin{prop}\label{prop:fin_ind_sbgp_overgp} Let $G$ be a group with a subgroup $H \leqslant G$ of index $k \in \N$. Then there are an overgroup $K$, of $G$,
with $|K:G|=2^k$, and an element $a \in K$, centralizing $H$, such that for any $x \in H$ and $y \in G$, $x \sim_H y$ if and only if $ax \sim_K ay$ in $K$.
\end{prop}

\begin{proof} Let $A=\Z/2\Z$ be the group of residues modulo $2$.
The natural action of $G$ on the left cosets modulo $H$ gives rise to the action of $G$ on the group $L\coloneq A^{G/H}$,
which can be thought of as the set of all functions from the set of left cosets $G/H$ to $A$, under addition.
The resulting semidirect product $K \coloneq L \rtimes G$ is the so-called \emph{permutational wreath product} of $A$ with $G$.
More explicitly, for every $f \in A^{G/H}$, thought of as a function $f:G/H \to A$, and any $g \in G$, we define
$gfg^{-1} \in A^{G/H}$ by the formula $(gfg^{-1})(uH) \coloneq f(g^{-1}uH)$, for all $uH \in G/H$.

Let $a \in L$ be the characteristic function of $H \in G/H$, that is $a(H)=\bar 1 \in A$ and $a(uH)=\bar 0$ if $uH \neq H$, where $A=\Z/2\Z=\{\bar 0,\bar 1\}$.
Clearly $|K:G|=|L|=2^{|G:H|}=2^k$ and $hah^{-1}=a$ for every $h \in H$, i.e., $a$ centralizes $H$.

Consider any $x \in H$ and $y \in G$. Evidently,  if $x \sim_H y$ then $ax \sim_K ay$, because $a$ centralizes $H$.
Conversely, assume that $ax \sim_K ay$. Then there are $b \in L$ and $g \in G$ such that
$bg (ay) g^{-1}b^{-1}=ax$ in $K$. Since $a,b \in L$ and $L$ is abelian, we get
$(gag^{-1})(gyg^{-1})=ab^{-1}xb=(ab^{-1}xbx^{-1})x$. As before, since $gag^{-1},ab^{-1}xbx^{-1} \in L $, $gyg^{-1},x \in G$ and $K$ is the semidirect product of $L$ and $G$,
we must have
\begin{equation}\label{eq:in_K}
gag^{-1}=ab^{-1}xbx^{-1} ~\mbox{ and }~ gyg^{-1}=x ~\mbox{ in }K.
\end{equation}

Suppose that $g \notin H$. Then $g^{-1}H \neq H$, so $(g a g^{-1})(H)=a(g^{-1}H)=\bar 0$ by the definition of $a$. On the other hand, since $x^{-1}H=H$, we
have the following equality in $A$:
\[(ab^{-1}xbx^{-1})(H)=a(H)+b^{-1}(H)+(xbx^{-1})(H)=\bar 1-b(H)+b(x^{-1}H)=\bar 1-b(H)+b(H)=\bar 1,\]
contradicting the first equation in \eqref{eq:in_K}. Therefore $g \in H$, and the second equation in \eqref{eq:in_K} yields that $x \sim_H y$.
This completes the proof of the proposition.
\end{proof}

Corollary \ref{cor:f_i_non-cs->f_oi_non-cs} from the Introduction can be deduced from Proposition \ref{prop:fin_ind_sbgp_overgp} in the same way as
Corollary \ref{cor:ind_2->overind_3} is deduced from Lemma \ref{lem:ind_2_sbgp->ind_3_overgp}.
Evidently one can draw a similar conclusion for the conjugacy problem:

\begin{cor}\label{cor:f_i_u_CP->f_oi_u_CP} Let $G$ be a finitely generated group possessing a subgroup of finite index with unsolvable conjugacy problem.
Then $G$ has a finite index overgroup with unsolvable conjugacy problem.
\end{cor}

\begin{rem} A theorem of Remeslennikov \cite[Thm. 1]{Rem} states that the restricted wreath product of two conjugacy separable groups is
conjugacy separable provided the base group is abelian and the acting group is cyclic subgroup separable. The argument from Proposition \ref{prop:fin_ind_sbgp_overgp}
shows that these conditions are no longer sufficient for conjugacy separability of a permutational wreath product (with finite orbits), because there exist conjugacy separable  and cyclic subgroup separable groups possessing non-conjugacy separable subgroups of finite index (see Theorem \ref{thm:norm_overgps-cs} or \cite[Thm. 1.1]{M-M}).
\end{rem}

\subsection{A non-hereditarily conjugacy separable group with conjugacy separable normal overgroups}
In this subsection we will prove Theorem \ref{thm:norm_overgps-cs} which is a stronger version of Theorem \ref{thm:norm_overgps-cs_simple} from the Introduction.
The proof will require several auxiliary statements.

\begin{lemma}\label{lem:characteristic} Suppose that $F_1,F_2$ are acylindrically
hyperbolic groups with cyclic centralizers and $G\leqslant \FF$ is a full subdirect product.
Let $N_i \coloneq G\cap F_i$, $i=1,2$. Then for every automorphism $\sigma \in Aut(G)$, either $\sigma(N_1)=N_1$ and $\sigma(N_2)=N_2$ or
 $\sigma(N_1)=N_2$ and $\sigma(N_2)=N_1$. Moreover, the latter is only possible if $F_1 \cong F_2$.
\end{lemma}

\begin{proof} Observe that for each $i=1,2$, $F_i$ cannot have non-trivial finite normal subgroups: the centralizer of such a normal
subgroup must have finite index in $F_i$ and it also must be cyclic, but $F_i$ is not virtually cyclic by definition.
Therefore $N_i\lhd F_i$ is non-elementary by Lemma~\ref{lem:inf_norm_sbgp->non-elem}, and hence it is non-abelian (as $F_i$ has cyclic centralizers).

Now, note that $N_2 \subseteq \C_G(N_1)$ and $N_1 \subseteq \C_G(N_2)$, so $\C_G(N_i)$ is non-abelian, $i=1,2$. However, for any element $(g_1,g_2) \in G$,
if $g_1 \neq 1$ and $g_2 \neq 1$ then $\C_G((g_1,g_2)) \leqslant \C_{F_1}(g_1) \times \C_{F_2}(g_2)$ is abelian, because $\C_{F_i}(g_i)$ are cyclic for $i=1,2$.
It follows that $N_1$ and $N_2$ are the only
maximal subgroups of $G$ with the property that $\C_G(N_i)$ is non-abelian. Hence any automorphism of $G$ either fixes both $N_1$ and $N_2$ or it interchanges them.

For the final claim, assume that $\sigma \in Aut(G)$ is an automorphism satisfying $\sigma(N_1)=N_2$. Then $\sigma$ naturally induces an isomorphism
between the quotients $G/N_1 \cong F_2$ and $G/N_2 \cong F_1$, sending $fN_1$ to $\sigma(f)N_2$, for all $fN_1 \in F/N_1$. Hence $F_1 \cong F_2$.
\end{proof}

The next statement was proved by Bumagina and Wise \cite{Bum-Wise} and is, in some sense, an amplification of  Rips's original construction \cite{Rips}.

\begin{lemma}\label{lem:Bum-Wise} For any finitely presented group $P$ and each integer $p>92$ there exist a group $F$, given by a
finite presentation $\langle Z \,\|\, \mathcal{R}\rangle$ satisfying the small cancellation condition $C'(1/11)$, and an epimorphism
$\psi:F \to P$ such that all of the following hold.
\begin{itemize}
  \item[(i)] There are $U,V \in Z$ such that $U^p, V^p \in \mathcal{R}$, and no other words in
  $\mathcal{R}$ are proper powers;
  \item[(ii)] $N\coloneq \ker\psi$ is generated by two elements $u,v \in F$ of order $p$, represented by the words $U,V$ respectively;
  \item[(iii)] $N$ is non-cyclic, infinite  and characteristic in $F$;
  \item[(iv)] the natural action of $F$ on $N$ by conjugation gives rise to a surjective homomorphism $P \to Out(N)$.
\end{itemize}
\end{lemma}

\begin{proof} Claims (i),(ii) and (iv) were proved in \cite[Lemma 9]{Bum-Wise} (that the orders of $u$ and $v$ are  exactly $p$
is an easy consequence of Greendlinger's lemma \cite[Thm. 4.5 in Sec. V.4]{L-S}).

The fact that $N$ is non-cyclic was noted in \cite[Lemma 10]{Bum-Wise}. Now, suppose that $N$ is finite. Then $\C_F(N)$ has finite index in $F$, but
$\C_F(N)=\{1\}$, as $F$ has cyclic centralizers (by Lemma~\ref{lem:small_canc->hyp}.(ii)) and $N$ is not cyclic. This implies that $F$ must  also be finite. However,
$|Z|\ge 2$ and no defining relator from $\mathcal R$ has length $1$ (by construction in \cite{Bum-Wise}), so the small cancellation group $F$ must be non-torsion
by \cite[Thm. VII]{Greend-torsion}. This contradiction shows that $|N|=\infty$.

Finally,  the fact that $N$ is characteristic
is an easy consequence of claims (i) and (ii). Indeed, (i), (ii)
together with  Lemma \ref{lem:small_canc->hyp}.(iii) show that every element of finite order is conjugate in $F$ to an element of $N$.
Since $N=\langle u,v \rangle$ and $u,v$ have order $p$, we can deduce that
$N$ is the normal closure of the torsion elements in $F$. The latter clearly implies that $N$ is characteristic in $F$.
\end{proof}

In the next lemma we observe a key property of centralizers in finite index normal overgroups of the group $F$, produced by the Bumagina-Wise construction from \cite{Bum-Wise},
which will be important in the proof of Theorem \ref{thm:norm_overgps-cs}.

\begin{lemma}\label{lem:centralizers_in_tilde_F}
Let $F$ be the group given by Lemma \ref{lem:Bum-Wise} (for some $P$ and $p$) and let $\tilde F$ be a normal overgroup of $F$, with
$|\tilde F:F|<\infty$. Then $\tilde F$ is hyperbolic and
for every element $f \in \tilde F$, either $|\tilde F:\C_{\tilde F}(f)|< \infty$ or $|\C_{\tilde F}(f):\langle f \rangle|<\infty$.
\end{lemma}

\begin{proof} Recall that $F$ that is hyperbolic by Lemma \ref{lem:small_canc->hyp}.(i), hence so is
$\tilde F$: since $|\tilde F:F|<\infty$, the natural inclusion of $F$ in $\tilde F$ induces a quasi-isometry between the
Cayley graphs of these groups (with respect to some finite generating sets), and hyperbolicity is preserved by quasi-isometries (see \cite[Thm. 1.9 in Ch. III.H]{B-H}).

Consider any $f \in \tilde F$. If $f$ has infinite order then $|\C_{\tilde F}(f):\langle f \rangle|<\infty$ by \cite[Cor. 3.10 in Ch.~III.$\Gamma$]{B-H}.
Thus we can now suppose that $f$ has finite order $n \in \N$ in $\tilde F$.

Let $N \lhd F$ be the normal subgroup from Lemma \ref{lem:Bum-Wise}. Since $N$ is characteristic in $F$ and $F \lhd \tilde F$, we deduce that $N \lhd \tilde F$.
Therefore conjugation by $f$ induces an automorphism of $N$. But then, according to Lemma \ref{lem:Bum-Wise}.(iv), there is an element $g \in F$
such that
\begin{equation}\label{eq:same_autom}
fhf^{-1}=ghg^{-1} ~\mbox{ for all } h \in N.
\end{equation}
It follows that $g^n h g^{-n}=f^n h f^{-n}=h$ for all $h \in H$, thus $g^n \in \C_{F}(N)$.
Note that $\C_{F}(N)=\{1\}$ as $N$ is not cyclic (by Lemma \ref{lem:Bum-Wise}.(iii)) and
$F$ has cyclic centralizers (by  Lemma \ref{lem:small_canc->hyp}.(ii)). Hence $g$ must have finite order in $F$.

Let $L\coloneq \C_{\tilde F}(g^{-1}f)$. Then, evidently,
\begin{equation}\label{eq:same_autom_in_L}
fhf^{-1}=ghg^{-1} ~\mbox{ for all } h \in L,
\end{equation}
and $N \subseteq L$ by  \eqref{eq:same_autom}.
It is well known that centralizers of elements are quasiconvex in any hyperbolic group (cf. \cite[Prop. 3.9 in Ch. III.$\Gamma$]{B-H}), therefore
$L$ is quasiconvex in $\tilde F$.
However, $L$ contains $N$, which is an infinite normal subgroup of $\tilde F$ by Lemma~\ref{lem:Bum-Wise}.(iii),
hence $|\tilde F:L|<\infty$ by \cite[Cor. 2]{Min-2}.

Now we need to consider two cases. If $g=1$ in $F$, then $L=\C_{\tilde F}(f)$ has finite index in $\tilde F$, as required.
Otherwise, $g \in F$ is a non-trivial element of finite order,
so $\C_{F}(g)$ is a finite cyclic group (as $F$ has cyclic centralizers), hence $|\C_{\tilde F}(g)|\le |\C_F(g)| \, |\tilde F:F|<\infty$. Recalling
\eqref{eq:same_autom_in_L}, we deduce that $\C_{\tilde F}(f) \cap L=\C_{\tilde F}(g) \cap L$ is finite, and so $|\C_{\tilde F}(f)|<\infty $ as $|\tilde F:L|<\infty$.
Consequently, $|\C_{\tilde F}(f):\langle f \rangle|<\infty$, and the lemma is proved.
\end{proof}

We are now ready to prove the main result of this section.
\begin{thm} \label{thm:norm_overgps-cs}
For each integer $k \ge 2$ there exists a finitely presented subdirect product
$G \leqslant \FF$, where $F_1, F_2$ are finitely presented $C'(1/11)$-groups, satisfying the following.
There is  a subgroup $G' \lhd G$, of index $k$, such that $G'$ is not conjugacy separable, but for every group
$K$, with $G \lhd K$ and $|K:G|<\infty$, $K$ is conjugacy separable.
\end{thm}

\begin{proof} It was shown in \cite[Example 6.1]{M-M}, using a result of Deligne \cite{Deligne},
that there is a finite index subgroup
$Q \leqslant {\rm Sp}(4,\Z)$ and a short exact sequence of groups
\[\{1\} \to O \to P \stackrel{\theta}{\to} Q \to \{1\}\]
such that $O \cong \Z/k\Z$ is central in $P$ and $P$ is not residually finite. Note that $Q$ has type $\mathrm{F}_3$ by the work of Borel and Serre \cite{Borel-Serre}
and $Q$ is cyclic subgroup separable as any subgroup of $\mathrm{GL}(4,\Z)$ (see \cite[Thm. 5 in Sec. 4.C]{Segal-book}).

Now, denote $p_1 \coloneq 93$ and $p_2\coloneq 94$. For each $i=1,2$, we can use Lemma \ref{lem:Bum-Wise} to find a $C'(1/11)$-small cancellation group $F_i$,
an epimorphism $\psi_i':F_i \to P$ and the normal subgroup $N_i'\coloneq \ker\psi_i'$, generated by two elements of order $p_i$,  from its claim.
Note that $F_1$ has an element of
order $p_1=93$, but the order of any torsion element in $F_2$ divides $94$ by Lemma~\ref{lem:small_canc->hyp}.(iii), hence $F_1 \not \cong F_2$.

The group $F_i$ is non-elementary (as it maps onto the non-elementary group $Q$) and hyperbolic (by Lemma \ref{lem:small_canc->hyp}.(i)), $i=1,2$. Moreover,
$F_i$ has cyclic centralizers by Lemma \ref{lem:small_canc->hyp}.(ii), and so it cannot have any non-trivial finite normal subgroups (as the centralizer of
such a subgroup would be cyclic and of finite index).

Set $\psi_i\coloneq \theta\circ \psi_i': F_i \to Q$, and let $G'\leqslant F_1\times F_2$ be the fibre product corresponding to  $\psi_1',\psi_2'$ and
$G\leqslant F_1\times F_2$ be the fibre product corresponding to  $\psi_1,\psi_2$.
Clearly, $G' \leqslant G$. Denote $N_i \coloneq G \cap F_i=\ker\psi_i \lhd F_i$, $i=1,2$, and observe that $N_1/N_1'\cong \ker\theta=O$ has order $k$,
thus $N_1=\bigsqcup_{j=1}^k s_j N_1'$, for some $s_1,\dots,s_{k} \in N_1$.
It is easy to see that $G=N_1G'$ (as $G' \leqslant \FF$ is subdirect and $N_1=G \cap F_1$),
which implies that $G=\bigsqcup_{j=1}^k s_jG'$, i.e., $|G:G'|=k$. The fact that $G'\lhd G$ easily follows from the fact that $O$ is central in $P$.
We can also deduce that $N_1$ is finitely generated, as this is true for $N_1'$,
hence $G$ is finitely presented by Lemma \ref{lem:1-2-3} (because $F_1/N_1 \cong Q$ is of type $\mathrm{F}_3$).

By Theorem \ref{thm:non-cs_for_subdir_of_hyp}, the group $G'$ is not conjugacy separable since $F_1/N_1'\cong P$ is not residually finite.

Now, suppose that  $K$ is a normal overgroup of $G$, with $|K:G|<\infty$. Since $F_1 \not\cong F_2$, Lemma~\ref{lem:characteristic} tells us that
$N_1,N_2 \lhd K$. Then, for every $i=1,2$, $\tilde F_i \coloneq K/N_i$ can be naturally considered as a normal overgroup of $F_i\cong G/N_i$, with
$|\tilde F_i/F_i|=|K/G|<\infty$. Note that since $N_1$ has trivial intersection with $N_2$, we can think of $K$ as a subdirect product in
$\tilde F_1 \times \tilde F_2$, with $K \cap \tilde F_i=N_i$, $i=1,2$ (see Subsection \ref{subsec:constr-subdir}).

We will now aim to apply Lemma \ref{lem:crit_for_CS} to show that $K$ is conjugacy separable. First we need to check that all the assumptions of this lemma are satisfied.
According to Lemma~\ref{lem:centralizers_in_tilde_F}, the groups $\tilde F_1$ and $\tilde F_2$ are hyperbolic.
Moreover, since $F_i$ is a group possessing  a finite presentation satisfying $C'(1/11)$, it is
virtually compact special (in the terminology of
Haglund and Wise \cite{H-W}) by a combination of the results of Wise \cite[Thm. 1.2]{Wise-small_canc-cubical} and Agol \cite[Thm. 1.1]{Agol}, $i=1,2$.
Therefore $\tilde F_1$  and $\tilde F_2$ are also virtually compact special, hence they must be hereditarily conjugacy separable by a theorem of the
author and Zalesskii \cite[Thm. 1.1]{M-Z-vcs}. Thus $\tilde F_1\times \tilde F_2$ is hereditarily conjugacy separable by Lemma \ref{lem:C-cs-for_products}.

Consider any element $(f_1,f_2) \in K$, where $f_i \in \tilde F_i$, $i=1,2$, and denote $C\coloneq \C_{\tilde F_1\times \tilde F_2}((f_1,f_2))$.
We need to check that the double coset $CK$ is closed in the profinite topology on $\tilde F_1\times \tilde F_2$.

First, assume that $|\tilde F_1:\C_{\tilde F_1}(f_1)|<\infty$. Then there is a finite index normal subgroup $L_1 \lhd F_1$, which is contained in $\C_{\tilde F_1}(f_1)$.
Since $K \leqslant \tilde F_1\times \tilde F_2$ is subdirect,  $T \coloneq L_1 K$ will be a finite index subgroup of $\tilde F_1 \times \tilde F_2$.
But $L_1 \leqslant C$, hence $CK=CL_1K=CT$ is equal to a union of left cosets modulo $T$. There are only finitely many of such cosets, so
$CK$ is closed in the profinite topology on $\tilde F_1\times \tilde F_2$, as a finite union of translates of $T$.

Obviously, if  $|\tilde F_2:\C_{\tilde F_2}(f_2)|<\infty$, we can show that $CK$ is closed in the profinite topology on $\tilde F_1\times \tilde F_2$ using a similar argument.
Thus, we can further suppose that $|\tilde F_i:\C_{\tilde F_i}(f_i)|=\infty$ for $i=1,2$.
Therefore $|\C_{\tilde F_i}(f_i):\langle f_i \rangle|<\infty$ for $i=1,2$, by Lemma~\ref{lem:centralizers_in_tilde_F}. This
implies that the subgroup
$H\coloneq \langle (f_1,1),(1,f_2) \rangle$ has finite index in $C$, and we can argue as in the proof of
Proposition \ref{prop:crit_of_CS_for_subdirect}. Indeed, suppose that $C=\bigcup_{j=1}^k (a_i,b_j)H$. Then, as $(f_1,f_2) \in K$ and
$H=\langle (f_1,1) \rangle \langle (f_1,f_2) \rangle$, we have $HK=\langle (f_1,1) \rangle K$. Consequently,
\begin{equation}\label{eq:CK}
CK=\bigcup_{j=1}^k (a_i,b_j)HK=\bigcup_{j=1}^k (a_i,b_j)\langle(f_1,1)\rangle K ~\mbox{ in } \tilde F_1\times\tilde F_2.
\end{equation}
Now, the double coset $\langle(f_1,1)\rangle K$ is closed in the profinite topology on $\tilde F_1\times \tilde F_2$ if and only if
$\langle f_1 \rangle N_1$ is closed in the profinite topology on $\tilde F_1$, by Lemma \ref{lem:crit_closed}, which, in its own turn, happens if and only if
the cyclic subgroup $\langle \psi(f_1)\rangle$ is closed in the profinite topology on $\tilde F_1/N_1$ (see Lemma \ref{lem:full_preimage}.(iii)). However, recall that
$F_1/N_1 \cong Q$ has finite index in $\tilde F_1/N_1$, and $Q$ is cyclic subgroup separable. Therefore $\tilde F_1/N_1$ is also cyclic subgroup separable
by Lemma~\ref{lem:cyc-sep-1}.(ii). Thus we conclude that $\langle(f_1,1)\rangle K$ is closed in the profinite topology on $\tilde F_1\times \tilde F_2$, which, in view of
\eqref{eq:CK}, implies that $CK$ is closed as well.

We have checked that $\tilde F_1 \times \tilde F_2$ and $K \leqslant \tilde F_1 \times\tilde F_2$ satisfy all the assumptions of Lemma  \ref{lem:crit_for_CS}. Therefore we
can use this lemma to deduce that $K$ is conjugacy separable. Thus the proof of the theorem is complete.
\end{proof}

\section{Conjugacy separability with respect to \texorpdfstring{$Q'$}{Q'}-groups} \label{sec:p}
This section investigates $\cC$-conjugacy separability of subdirect products when $\cC$ is a class of $p$-groups, or, more generally, a class of $Q'$-groups.

\begin{defn}\label{def:q'} Let $Q \subset \N$ be a set of prime numbers and let $F$ be a group.
We will say that $F$ is a \emph{$Q'$-group} if every element of $F$ has finite order which is coprime to each $q \in Q$.
\end{defn}

If $p$ is a prime then the class of $p$-groups is precisely the class of all $Q'$-groups, where $Q\coloneq \mathbb{P}\setminus \{p\}$ and $\mathbb{P}$
denotes the set of all prime numbers.

It is easy to see that for any $Q \subseteq \mathbb{P}$ the class of \emph{all} $Q'$-groups is an extension-closed pseudovariety (which is trivial if and only if
$Q=\mathbb{P}$), and every group in this class is periodic.
By a \emph{pseudovariety of $Q'$-groups} we will mean a pseudovariety which consists only of $Q'$-groups (but it does not have to contain all $Q'$-groups).

The following theorem is an improvement of Corollary \ref{cor:non-cs_for_subdir_of_acyl_hyp} in the case when $\cC$ is a class of $Q'$-groups.

\begin{thm}\label{thm:q'-crit} Let $Q \subseteq \mathbb{P}$ be a non-empty set of primes and let $\cC$ be a pseudovariety of $Q'$-groups.
Suppose that $F_1,F_2$ are acylindrically hyperbolic groups without
non-trivial finite normal subgroups, $G \leqslant \FF$ is a full subdirect product and $N_1 \coloneq G\cap F_1$. If $G$ is $\cC$-conjugacy separable then
$F_1/N_1$ is a residually-$\cC$ $Q'$-group.
\end{thm}

The proof of Theorem \ref{thm:q'-crit} will employ the following lemma.

\begin{lemma}\label{lem:non_cs_for q'} Suppose that $Q \subseteq \mathbb{P}$ is a set of primes and
$\cC$ is a pseudovariety of $Q'$-groups. Let $F_1,F_2$ be groups, let $G \leqslant F_1\times F_2$
be a subgroup such that $\rho_1(G)=F_1$, where $\rho_1:\FF \to F_1$ is the natural projection, and let $N_1\coloneq G \cap F_1$.
If  $x_1 \in F_1$ and $(y_1,y_2) \in G$ are elements such that $y_1 \in N_1 x_1^q$, for some $q \in Q$, then $(x_1y_1x_1^{-1},y_2) \in G$ belongs to the closure of the $N_1$-conjugacy class
$(y_1,y_2)^{N_1}\subseteq (y_1,y_2)^G$ in the \proc on $G$.
\end{lemma}

\begin{proof} First, observe that $N_1 \lhd F_1$, as $\rho_1(G)=F_1$, and, since there is $h \in N_1$ such that $y_1=h x_1^q$, for each $k \in \Z$ we have
\[(x_1^k y_1x_1^{-k},y_2)=(x_1^k hx_1^{-k}h^{-1}y_1,y_2)=(x_1^k hx_1^{-k} h^{-1},1)(y_1,y_2) \in N_1(y_1,y_2) \subseteq G.\]
Thus $(x_1^k y_1x_1^{-k},y_2) \in G$ for all $k \in \Z$.
We also see that for each $n \in \Z$ there is $h_1 \in N_1$ such that $x_1^{nq}= h_1 y_1^n$, hence
\begin{equation}\label{eq:x_1^nq}
(x_1^{nq}y_1x_1^{-nq},y_2)=(h_1 y_1 h_1^{-1},y_2)=(h_1,1)(y_1,y_2)(h_1,1)^{-1} \sim_{N_1} (y_1,y_2)~\mbox{ in } G.
\end{equation}

By the assumptions, there is $x_2 \in F_2$ such that $(x_1,x_2) \in G$.
Let $M \in \cC$ be any group and let $\varphi:G \to M$ be a homomorphism. Then the order of $\varphi((x_1,x_2))$ in $M$ is some $l \in \N$ which is coprime to $q$,
because $M$ is a $Q'$-group and $q \in Q$.

Set $y_1'\coloneq x_1 y_1x_1^{-1} \in F_1$. Let us now show that for every integer $m \in \Z$ we have
\begin{equation}\label{eq:x_1^ml}
\varphi\left((x_1^{ml} y_1'x_1^{-ml},y_2)\right)=\varphi\left((y_1',y_2)\right) ~\mbox{ in } M.
\end{equation}
Indeed, clearly $y_1'=gx_1^q$, where $g\coloneq x_1hx_1^{-1} \in N_1$. Since the elements $(y_1',y_2)$, $(x_1,x_2)$ and $(g,1)$ all belong to $G$
and $\varphi\left( (x_1,x_2)\right)^{ml}=1$ in $M$, we obtain

\begin{multline*}
\varphi\left((x_1^{ml} y_1' x_1^{-ml},y_2)\right)=\varphi\left((x_1^{ml} g x_1^{-ml} g^{-1}y_1',y_2)\right) =\varphi\left((x_1^{ml} g x_1^{-ml} g^{-1},1)(y_1',y_2)\right)
\\ = \varphi\left( (x_1,x_2)^{ml} (g,1) (x_1,x_2)^{-ml} (g,1)^{-1}  (y_1',y_2)\right)  \\ =
\varphi\left( (x_1,x_2)\right)^{ml} \varphi \left((g,1)\right) \varphi\left((x_1,x_2)\right)^{-ml} \varphi\left((g,1)  \right)^{-1} \varphi\left( (y_1',y_2)\right)
=\varphi\left( (y_1',y_2)\right).
\end{multline*}
Thus we have established the validity of equation \eqref{eq:x_1^ml}.

Finally, since $q$ and $l$ are coprime, there exist $m,n \in \Z$ such that $nq=ml+1$. Therefore we can combine \eqref{eq:x_1^nq} with \eqref{eq:x_1^ml} to achieve
\[\varphi\left((y_1,y_2)\right) \sim_{\varphi(N_1)} \varphi\left((x_1^{nq}y_1x_1^{-nq},y_2)\right)=
\varphi\left((x_1^{ml} y_1'x_1^{-ml},y_2)\right)=\varphi\left((y_1',y_2)\right) ~\mbox{ in } M.\]
Thus $\varphi\left((y_1,y_2)\right) \sim_{\varphi(N_1)} \varphi\left((y_1',y_2)\right)$ in $M$. Since $M \in \cC$ is arbitrary, we can conclude that
$(y_1',y_2)=(x_1 y_1x_1^{-1},y_2)$ belongs to the closure of $(y_1,y_2)^{N_1}$ in the \proc on $G$, as claimed.
\end{proof}

\begin{rem} \label{rem:res-C->fin_sbgps_in_C} If $\cC$ is a pseudovariety of groups and $P$ is a residually-$\cC$ group then every finite subgroup of $P$ belongs to $\cC$.
In particular, if $P$ is periodic and $\cC$ consists of $Q'$-groups, for some $Q \subseteq \mathbb{P}$, then $P$ is itself a $Q'$-group.
\end{rem}

Indeed, if $P$ is residually-$\cC$, then every finite subgroup $A \leqslant P$ is $\cC$-closed in $P$, hence it injects into some quotient $M \in \cC$, of $G$. Therefore
$A \in \cC$ as $\cC$ is closed under taking subgroups.

\begin{proof}[Proof of Theorem \ref{thm:q'-crit}] By Corollary \ref{cor:non-cs_for_subdir_of_acyl_hyp}, $F_1/N_1$ must be residually-$\cC$, and so, in view of
Remark \ref{rem:res-C->fin_sbgps_in_C}, it remains to
prove that this group is periodic. Arguing by contradiction, suppose that there is an element $\bar x \in F_1/N_1$ of infinite order, and let $x_1 \in F_1$ be any preimage of $\bar x$ in $F_1$.

Choose some $q \in Q$. By the assumptions, the normal subgroups $N_1 \lhd F_1$ and $N_2 \coloneq G \cap F_2\lhd F_2$ must be infinite.
Therefore, according to Lemma \ref{lem:inf_norm_sbgp->non-elem}, there is $h_1 \in N_1$ and $m \in \Z$
such that $y_1 \coloneq h_1^m x_1^q \in N_1x_1^q \subseteq F_1$ satisfies $\C_{F_1}(y_1)=\langle y_1\rangle$. Take $v \in F_2$ so that $(y_1,v) \in G$, and
apply Lemma \ref{lem:inf_norm_sbgp->non-elem} again, to find $h_2 \in N_2$ and $n \in \Z$ such that the element $y_2\coloneq h_2^n v \in F_2$ satisfies
$\C_{F_2}(y_2)=\langle y_2\rangle$. Observe that $(y_1,y_2) = (1,h_2^n)(y_1,v) \in N_2 G=G$.

By Lemma \ref{lem:non_cs_for q'}, the element $(x_1y_1x_1^{-1},y_2)\in G$ belongs to the closure of $(y_1,y_2)^G$ in the \proc on $G$. Let us now check that
$(x_1y_1x_1^{-1},y_2) \notin (y_1,y_2)^G$. Indeed, since $(x_1y_1x_1^{-1},y_2)=(x_1,1)(y_1,y_2)(x_1,1)^{-1}$, this element is conjugate to $(y_1,y_2)$ in $G$
if and only if $(x_1^{-1},1) \in \C_{\FF}((y_1,y_2)) G$ by Remark~\ref{rem:in_cc->double_coset}. But
\[\C_{\FF}((y_1,y_2))=\C_{F_1}(y_1) \times \C_{F_2}(y_2)=\langle y_1 \rangle \times \langle y_2 \rangle=\langle (y_1,1)\rangle \langle(y_1,y_2)\rangle,\]
hence $\C_{\FF}((y_1,y_2)) G=\langle (y_1,1) \rangle G$, as $(y_1,y_2) \in G$. It follows from Remark \ref{rem:YG_int_F_1}, that
$(x_1^{-1},1) \in \C_{\FF}((y_1,y_2)) G$ if and only if $x_1^{-1} \in \langle y_1 \rangle N_1$ in $F_1$. Since $y_1N_1=x_1^qN_1$,
the latter is equivalent to $\bar x ^{-1} \in \langle \bar x^q \rangle $ in $F_1/N_1$, which is impossible as $q \ge 2$ and $\bar x$ has infinite order in $F_1/N_1$. Thus
$(x_1^{-1},1) \notin \C_{\FF}((y_1,y_2)) G$, implying that $(x_1y_1x_1^{-1},y_2) \notin (y_1,y_2)^G$. This means that $G$ is not $\cC$-conjugacy separable,
contradicting our assumption. So we can conclude that every element in $F_1/N_1$ must have finite order.

Finally, since $F_1/N_1$ is residually-$\cC$ and periodic, we can finish the proof of the theorem by using
Remark \ref{rem:res-C->fin_sbgps_in_C} to deduce that $F_1/N_1$ is a $Q'$-group.
\end{proof}

\begin{cor}\label{cor:Q'-existence} Let $\cC$ be a pseudovariety of $Q'$-groups, for some non-empty subset $Q\subseteq \mathbb{P}$. Suppose that there
exists a finitely generated subdirect product $G \leqslant \FF$ such that $F_1$, $F_2$ are finitely presented acylindrically hyperbolic groups
without non-trivial finite normal subgroups, $G$ is $\cC$-conjugacy separable and $|(\FF):G|=\infty$. Then there exists a finitely presented
infinite residually-$\cC$ $Q'$-group.
\end{cor}

\begin{proof} This statement is an immediate consequence of Lemma~\ref{lem:Mih}.(b), Lemma~\ref{lem:norm_in_G->norm_in_F}.(iii) and Theorem~\ref{thm:q'-crit}.
\end{proof}

For example, let $p$ be a prime and $\cC$ be the class of all $p$-groups (including the infinite ones). If there is a finitely generated subdirect product
$G \leqslant \FF$, of non-abelian free groups $F_1$, $F_2$, such that $G$ is $\cC$-conjugacy separable and $|(\FF):G|=\infty$ then,
according to Corollary~\ref{cor:Q'-existence}, there exists a finitely presented infinite $p$-group.

In Subsection \ref{subsec:necessity} we discussed the gap between the sufficient criterion for $\cC$-conjugacy separability of subdirect products given by
Proposition~\ref{prop:crit_of_CS_for_subdirect} and the necessary criterion provided by Theorem \ref{thm:first_crit}. Theorem \ref{thm:q'-crit} allows us to
close this gap in the case when $\cC$ consists of $Q'$-groups, for some non-empty $Q \subseteq \mathbb{P}$, because
a periodic group is cyclic subgroup $\cC$-separable if and only if it is residually-$\cC$.

\begin{thm} \label{thm:q'-iff_crit} Let $Q \subseteq \mathbb{P}$ be a non-empty set of primes and let $\cC$ be an extension-closed pseudovariety of finite $Q'$-groups.
Suppose that $F_1,F_2$ are $\cC$-hereditarily conjugacy separable acylindrically hyperbolic groups with cyclic centralizers and
without non-trivial finite normal subgroups, and $G \leqslant \FF$ is a full
subdirect product. Then the following are equivalent:
\begin{itemize}
  \item[(a)] $G$ is $\cC$-conjugacy separable;
  \item[(b)] $F_1/N_1$ is a residually-$\cC$ $Q'$-group, where $N_1 \coloneq G \cap F_1$;
  \item[(c)]  $F_1/N_1$ is a residually-$\cC$  periodic group.
\end{itemize}
\end{thm}

\begin{proof} The fact that (a) implies (b) is given by Theorem \ref{thm:q'-crit}, and   (b) implies (c) by Definition~\ref{def:q'}.
If (c) holds, then $F_1/N_1$ is a residually-$\cC$ periodic group,
so every cyclic subgroup is finite and, hence, $\cC$-closed in $F_1/N_1$. Therefore we can deduce (a) from Proposition \ref{prop:crit_of_CS_for_subdirect}.
\end{proof}

We will now focus on the applications of the above theorem in the case when $\cC=\cC_p$ is the class of finite $p$-groups.
In this case instead of writing residually-$\cC_p$ we will write \emph{residually-$p$}.
Let us first prove Corollary \ref{cor:p_cs-iff_crit-simple} from the Introduction.

\begin{cor}\label{cor:p_cs-iff_crit} Suppose that $p$ is a prime, $G \leqslant \FF$ is a full subdirect product of non-abelian free groups and $N_1\coloneq G \cap F_1$. Then
the following are equivalent:
\begin{itemize}
  \item[(1)] $G$ is $p$-conjugacy separable;
  \item[(2)] $F_1/N_1$ is a residually-$p$ periodic group;
  \item[(3)] $F_1/N_1$ is a residually finite $p$-group.
\end{itemize}
\end{cor}

\begin{proof} Let $Q\coloneq \mathbb{P}\setminus\{p\}$, then the class of all finite $p$-groups coincides with the class of all finite $Q'$-groups. Therefore
(1) is equivalent to (2) by Theorem \ref{thm:q'-iff_crit} (recall that $F_1$ and $F_2$ are $p$-hereditarily conjugacy separable by Lemma \ref{lem:free_are_C-hcs}).
And (2) is equivalent to (3) because a periodic group is residually-$p$ is and only if it is
a residually finite $p$-group.
\end{proof}

\begin{ex}\label{ex:fibre-Z}
Let $F$ be the free group of rank $2$, let $\psi:F \to \Z$ be any epimorphism and let $G \leqslant F\times F$ be the corresponding symmetric fibre product. Then
 $G$ is finitely generated (Lemma \ref{lem:Mih}.(a)) and hereditarily conjugacy separable  (Corollary \ref{cor:crit_for_hcs_of_subdir}), but it is not $p$-conjugacy separable
for any prime $p$ by Corollary \ref{cor:p_cs-iff_crit}, as $\Z$ is not a $p$-group.

Moreover, arguing as in Lemma \ref{lem:C-open-subdir}, we see that
if $H\leqslant G$ is any finite index subgroup then $H\leqslant J_1 \times J_2$ is a subdirect product of some finite index subgroups $J_1,J_2\leqslant F$, so
that $J_1/(H\cap J_1)$ maps onto $J_1/(J_1\cap \ker\psi)$, which has finite index in $F/\ker\psi \cong \Z$. It follows that $J_1/(H\cap J_1)$ cannot be a periodic group, hence
$H$ is not $p$-conjugacy separable for any prime $p$ by Corollary \ref{cor:p_cs-iff_crit}.
Thus $G$ is not even virtually $p$-conjugacy separable, for any $p \in \mathbb{P}$.
\end{ex}

\begin{ex} Let $P$ be the first Grigorchuk's group \cite{Grig-1}.
This is an infinite residually finite $2$-group generated by $3$ elements (cf. \cite[Thm.]{Grig-1} and \cite[Prop. 6 and Remark 11 in Ch. VIII]{Harpe}).
Therefore
there is an epimorphism $\psi:F \to P$, where $F$ is the free group of rank $3$, and we can construct the corresponding symmetric fibre product $G \leqslant F\times F$.

By Corollary \ref{cor:p_cs-iff_crit}, $G$ is $2$-conjugacy separable. Note that $G$  has infinite index in $F \times F$ because $P$ is infinite
(cf. Lemma \ref{lem:norm_in_G->norm_in_F}.(iii)). Moreover, $G$ is not finitely generated by Lemma~\ref{lem:Mih}.(b) because $P$ is not finitely presented
(see \cite[Thm.]{Grig-1} or \cite[Thm. 55 in Sec. VIII.E]{Harpe}).
\end{ex}

Let us now prove Corollary \ref{cor:equiv_to_ex_of_fp_p-gp} mentioned in the Introduction.
\begin{proof}[Proof of Corollary \ref{cor:equiv_to_ex_of_fp_p-gp}]
First, suppose that (1) holds and let $P$ be an infinite finitely presented residually finite $p$-group. Let $F$ be a finitely generated free group possessing an epimorphism
$\psi:F \to P$, and let $G\leqslant F\times F$ be the corresponding symmetric fibre product. Since $F$ is finitely generated, it can be embedded into the free
group $H$ of rank $2$, so $G \leqslant H\times H$. Since $P$ is finitely presented and infinite, $G$ is finitely generated by Lemma~\ref{lem:Mih}.(a)
and has infinite index in $F\times F$ by Lemma \ref{lem:norm_in_G->norm_in_F}.(iii). The latter implies that $G$ cannot be finitely presented by a result of
Baumslag and Roseblade \cite[Thm. B]{Baum-Rose}, as $P$ is not free.
Therefore $G$ is not virtually a direct product of free groups (the free groups would have to be finitely generated
as $G$ is finitely generated). Finally, $G$ is $p$-conjugacy separable by Corollary \ref{cor:p_cs-iff_crit}, because $P$ is a residually finite $p$-group.
Hence (2) holds.

Now let us show that (2) implies (3). Let $F_i$ be the projection of $G$ to the $i$-th coordinate group, $i=1,2$.
Then $F_i$ is a finitely generated free group, $i=1,2$,
and $G$ can be considered as a subdirect product in $\FF$. If $G \cap F_i=\{1\}$ for some $i \in \{1,2\}$, then $G$ is free, contradicting our assumption. Therefore
$G$ is a full subdirect product in $\FF$. Now, if $|(\FF):G|<\infty$ then $F_1/N_1 \cong G/(N_1\times N_2)$ would be finite
(see claims (ii) and (iii) of Lemma~\ref{lem:norm_in_G->norm_in_F}),
where $N_i\coloneq G \cap F_i$, $i=1,2$. This would yield that the direct product of the free groups $N_1$ and $N_2$ has finite index in $G$, which
is again impossible by the assumptions of (2). Thus $G$ has infinite index in $\FF$, so (3) holds.

Finally, let $G \leqslant \FF$ be the full subdirect product satisfying claim (3) and let $N_i \coloneq G \cap F_i$, $i=1,2$. If $F_i$
was cyclic for some $i \in \{1,2\}$ then $|F_i:N_i|<\infty$, as $N_i \neq \{1\}$ by the assumption, which would imply that
$|(\FF):G|<\infty$ by Lemma \ref{lem:norm_in_G->norm_in_F}. The latter would contradict another assumption of (3), hence $F_1$ and $F_2$ must both be non-abelian.

We can now apply Lemma~\ref{lem:norm_in_G->norm_in_F}.(iii), Lemma~\ref{lem:Mih}.(b) and Corollary~\ref{cor:p_cs-iff_crit}  to conclude that
$F_1/N_1$ is an infinite finitely presented
residually finite $p$-group. Thus (3) implies (1).
\end{proof}

Our last two corollaries show that $p$-conjugacy separable subgroups in direct products of two free groups are very rare. The first one is an immediate consequence of
Theorem~\ref{thm:q'-crit} and the fact that no non-trivial group can be a $p$-group for two distinct primes $p$.
\begin{cor}\label{cor:p-cs_for_2_p's-subdir} Suppose that $G \leqslant \FF$ is a full subdirect product of acylindrically hyperbolic groups $F_1$ and $F_2$, which do not have
non-trivial finite normal subgroups. If $G$ is $p$-conjugacy separable for at least two distinct primes $p$, then $G=F_1\times F_2$.
\end{cor}

Corollary \ref{cor:p-cs_for_2_p's-general} from the Introduction treats the general case of arbitrary subgroups in direct products of two free groups.

\begin{proof}[Proof of Corollary \ref{cor:p-cs_for_2_p's-general}]
The sufficiency is clear, as the direct product of two free groups is $p$-conjugacy separable for any $p \in \mathbb{P}$ by
Lemmas \ref{lem:free_are_C-hcs} and \ref{lem:C-cs-for_products}.

Thus it remains to establish the necessity. So, suppose that $G$ is a subgroup in $\FF$, where $F_1$, $F_2$ are free groups, and $G$ is $p$-conjugacy separable
for at least two distinct primes $p$. As before, without loss of generality, we can assume that $G$ is  subdirect in $\FF$.
If $G \cap F_i=\{1\}$, for some $i \in \{1,2\}$, then $G$ is free. Otherwise, $G\leqslant \FF$ is a full subdirect product.

If $F_1\cong \Z$, then $N_1 \coloneq G \cap F_1$ is infinite cyclic and central in $\FF$.
Moreover, $G/N_1 \cong F_2$ is free, therefore $G$ is a split extension of $N_1$ by a free subgroup, isomorphic to $F_2$. It follows that
$G \cong N_1 \times F_2 \cong \Z \times F_2$,
as $N_1$ is central.

Similarly, if $F_2 \cong \Z$, we can show that $G\cong F_1\times \Z$. Thus we can further suppose that $F_1$ and $F_2$ are non-abelian.
In this case all the assumptions
of Corollary \ref{cor:p-cs_for_2_p's-subdir} are satisfied, which yields that $G=F_1\times F_2$.
\end{proof}


\begin{thebibliography}{99}
\bibitem{Agol} I. Agol,
The virtual Haken conjecture. With an appendix by I. Agol, D. Groves and J. Manning.
{\it Documenta Math}. \textbf{18} (2013), 1045--1087.


\bibitem{A-M-S} Y. Antol\'{i}n, A. Minasyan, A. Sisto, Commensurating endomorphisms of acylindrically hyperbolic groups and applications. \emph{Groups, Geom. Dyn.}
\textbf{10} (2016), no. 4, 1149--1210.

\bibitem{B-L-S} H. Bass, M. Lazard, J.-P. Serre, Sous-groupes d'indice fini dans $\mathrm{SL}(n,\Z)$. (French)
\emph{Bull. Amer. Math. Soc.} \textbf{70} (1964), 385--392.


\bibitem{B-gp} G. Baumslag, A non-cyclic one-relator group all of whose finite quotients are cyclic.
\emph{J. Austral. Math. Soc.} \textbf{10} (1969), no. 3-4, 497--498.

\bibitem{BBMS:1-2-3_sym}  G. Baumslag, M.R. Bridson, C.F. Miller III, H. Short, Fibre products, non-positive curvature, and decision problems.
\emph{Comment. Math. Helv.} \textbf{75} (2000), no. 3, 457--477.

\bibitem{Baum-Rose} G. Baumslag, J.E. Roseblade,
Subgroups of direct products of free groups. \emph{J. London Math. Soc. (2)} \textbf{30} (1984), no. 1, 44--52.


\bibitem{B-H} M.R. Bridson, A. Haefliger, \emph{Metric spaces of non-positive curvature}. Grundlehren der Mathematischen
Wissenschaften [Fundamental Principles of Mathematical Sciences], 319. \emph{Springer-Verlag, Berlin}, 1999.
xxii+643 pp.

\bibitem{BHMS:1-2-3_asym} M.R. Bridson, J. Howie, C.F. Miller III, H. Short,
On the finite presentation of subdirect products and the nature of residually free groups.
\emph{Amer. J. Math.} \textbf{135} (2013), no. 4, 891--933.



\bibitem{B-F} M. Bestvina, K. Fujiwara, Bounded cohomology of subgroups of mapping class groups. \emph{Geom. Topol.} \textbf{6}(2002), 69--89.

\bibitem{Borel-Serre} A. Borel, J.-P. Serre, Cohomologie d'immeubles et de groupes S-arithm\'etiques. (French) \emph{Topology} \textbf{15} (1976),
no. 3, 211--232.

\bibitem{Bum-Wise} I. Bumagin, D.T. Wise, Every group is an outer automorphism group of a finitely generated group. \emph{J. Pure Appl. Algebra} \textbf{200} (2005),
no. 1-2, 137--147.

\bibitem{Bur-Mar} J. Burillo, A. Martino, Quasi-potency and cyclic subgroup separability. \emph{J. Algebra} \textbf{298} (2006), no. 1, 188--207.

\bibitem{Cap-Min} P.-E. Caprace, A. Minasyan, On conjugacy separability of some Coxeter groups and parabolic-preserving
automorphisms. \emph{Illinois J. Math.} \textbf{57} (2013), no. 2,  499--523.


\bibitem{Chag-Zal-fp_res_free} S.C. Chagas, P.A. Zalesskii, Finite index subgroups of conjugacy separable groups.
\emph{Forum Mathematicum} \textbf{21} (2009), no. 2, 347--353.

\bibitem{Chag-Zal-lim} S.C. Chagas, P.A. Zalesskii, Limit groups are conjugacy separable. \emph{Internat. J. Algebra Comput.} \textbf{17}
(2007), no. 4, 851--857.



\bibitem{Col-Mil} D.J. Collins, C.F. Miller III, The conjugacy problem and subgroups of finite index.
\emph{Proc. London Math. Soc. (3)} \textbf{34} (1977), no. 3, 535--556.

\bibitem{Cornulier}  Y. de Cornulier, Finitely presented wreath products and double coset decompositions. \emph{Geom. Dedicata} \textbf{122} (2006), 89--108.

\bibitem{DGO} F. Dahmani, V. Guirardel and D. Osin, Hyperbolically embedded subgroups and rotating families in groups
acting on hyperbolic spaces. \emph{Mem. Amer. Math. Soc.} \textbf{245} (2017), no. 1156, v+152 pp.

\bibitem{Deligne} P. Deligne, Extensions centrales non r\'esiduellement finies de groupes arithm\'etiques. (French) \emph{C. R. Acad. Sci.
Paris S\'er. A-B} \textbf{287} (1978), no. 4, A203--A208.



\bibitem{Dison}  W. Dison, Isoperimetric functions for subdirect products and Bestvina-Brady groups. Ph.D. thesis, Imperial
College London, 2008. \texttt{arXiv:0810.4060}

\bibitem{Dyer} J.L. Dyer, Separating conjugates in free-by-finite groups. \emph{J. London Math. Soc. (2)} \textbf{20} (1979), no. 2, 215--221.




\bibitem{Ferov-1} M. Ferov, On conjugacy separability of graph products of groups. \emph{J. Algebra} \textbf{447} (2016), 135--182.

\bibitem{Ferov-thesis} M. Ferov,  Separability properties of graph products of groups.
Doctoral Thesis, \emph{University of Southampton, School of Mathematics}, 2015, 116pp. Available from \url{http://eprints.soton.ac.uk/384000/}

\bibitem{Formanek} E. Formanek, Conjugate separability in polycyclic groups.
\emph{J. Algebra} \textbf{42} (1976), no. 1, 1--10.




\bibitem{Gor} A.V. Goryaga, Example of a finite extension of an FAC-group that is not an FAC-group. (Russian)
\emph{Sibirsk. Mat. Zh.} \textbf{27} (1986), no. 3, 203--205.

\bibitem{Gor-Kirk} A.V. Goryaga,  A.S. Kirkinskii, Decidability of the conjugacy problem does not carry over to finite extensions of groups
(aka. The decidability of the conjugacy problem cannot be transferred to finite extensions of groups). (Russian)
\emph{Algebra i Logika} \textbf{14} (1975), no. 4, 393--406. English translation in \emph{Algebra and Logic} \textbf{14} (1975), no. 4, 240--248.


\bibitem{Greend-torsion} M. Greendlinger, On Dehn’s algorithms for the conjugacy and word problems, with applications.
\emph{Comm. Pure Appl. Math.} \textbf{13} (1960), 641--677.

\bibitem{Grig-1} R.I. Grigorchuk, On Burnside's problem on periodic groups. (Russian)
\emph{Funktsional. Anal. i Prilozhen.} \textbf{14} (1980), no. 1, 53--54. English translation in \emph{Functional Anal. Appl.} \textbf{14} (1980), no. 1, 41--43.


\bibitem{Grossman} E.K. Grossman, On the residual finiteness of certain mapping class groups.
\emph{J. London Math. Soc. (2)} \textbf{9} (1974/75), 160--164.

\bibitem{H-W} F. Haglund, D.T. Wise, Special cube complexes.
\emph{Geom. Funct. Anal.} \textbf{17} (2008), no. 5, 1551--1620.

\bibitem{Hall} P. Hall, On the finiteness of certain soluble groups.
\emph{Proc. London Math. Soc. (3)} \textbf{9} (1959), 595--622.

\bibitem{H-W-Z} E. Hamilton, H. Wilton, P.A. Zalesskii, Separability of double cosets and conjugacy classes in $3$-manifold groups.
\emph{J. Lond. Math. Soc. (2)} \textbf{87} (2013), no. 1, 269--288.

\bibitem{Harpe}  P. de la Harpe, Topics in geometric group theory. Chicago Lectures in Mathematics. \emph{University of Chicago Press, Chicago, IL}, 2000. vi+310 pp.

\bibitem{Kul} O.V. Kulikova, On the conjugacy problem in the group $F/N_1\cap N_2$. (Russian) \emph{Mat. Zametki} \textbf{93} (2013), no. 6, 853--868.
English translation in  \emph{Math. Notes} \textbf{93} (2013), no. 5-6, 837--849.

\bibitem{Lor} K. Lorensen, Groups with the same cohomology as their profinite completions.
\emph{J. Algebra} \textbf{320} (2008), no. 4, 1704-1722.


\bibitem{L-S} R.C. Lyndon, P.E. Schupp, Combinatorial group theory. Ergebnisse der Mathematik und ihrer Grenzgebiete,
Band 89. \emph{Springer-Verlag, Berlin-New York}, 1977. xiv+339 pp.

\bibitem{Malcev} A.I. Mal'cev, On homomorphisms onto finite groups (Russian). \emph{Uchen. Zap. Ivanovskogo Gos. Ped. Inst.} \textbf{18} (1958), 49--60.
English translation in: \emph{Amer. Math. Soc. Transl. Ser. (2)} \textbf{119} (1983) 67-79.

\bibitem{Malcev-lin_gps} A.I. Mal'cev, On isomorphic matrix representations of infinite groups. (Russian)
\emph{Rec. Math. [Mat. Sbornik] N.S.} \textbf{8 (50)} (1940), no. 3, 405--422.

\bibitem{Martino} A. Martino, A proof that all Seifert $3$-manifold groups and all virtual surface groups are conjugacy separable.
\emph{J. Algebra} \textbf{313} (2007), no. 2, 773--781.

\bibitem{M-M} A. Martino, A. Minasyan, Conjugacy in normal subgroups of hyperbolic groups. \textit{Forum Math.} \textbf{24} (2012), no. 5, 889--909.


\bibitem{Mih} K.A. Miha{\u\i}lova, The occurrence problem for direct products of groups. (Russian)
\emph{Mat. Sb. (N.S.)} \textbf{70 (112)} (1966), no. 2, 241--251.

\bibitem{Miller-book}  C.F. Miller III, On group-theoretic decision problems and their classification.
Annals of Mathematics Studies, No. 68. \emph{Princeton University Press, Princeton, N.J.; University of Tokyo Press, Tokyo,} 1971. viii+106 pp.

\bibitem{M-RAAG} A. Minasyan, Hereditary conjugacy separability of right angled Artin groups and its applications. \emph{Groups Geom. Dyn.} \textbf{6} (2012), no. 2, 335--388.

\bibitem{Min-2} A. Minasyan, Some properties of subsets of hyperbolic groups. \emph{Comm. Algebra} \textbf{33} (2005), no. 3, 909--935.

\bibitem{M-O_acyl_trees} A. Minasyan, D. Osin, Acylindrical hyperbolicity of groups acting on trees.
\emph{Math. Ann.} \textbf{362} (2015), no. 3-4, 1055--1105.

\bibitem{M-Z-1_rel} A. Minasyan, P. Zalesskii, One-relator groups with torsion are conjugacy separable.
\emph{J. Algebra} \textbf{382} (2013), 39--45.

\bibitem{M-Z-vcs} A. Minasyan, P. Zalesskii, Virtually compact special hyperbolic groups are conjugacy separable. \emph{Comm. Math. Helv.},
\textbf{91} (2016), no. 4, 609--627.

\bibitem{Mostowski} A.W. Mostowski, On the decidability of some problems in special classes of groups.
\emph{Fund. Math.} \textbf{59} (1966), 123--135.

\bibitem{Munkres}  J.R. Munkres, Topology: a first course. \emph{Prentice-Hall, Inc., Englewood Cliffs, N.J.}, 1975. xvi+413 pp.

\bibitem{Newman} B.B. Newman, Some results on one-relator groups. \emph{Bull. Amer. Math. Soc.} \textbf{74} (1968), 568--571.

\bibitem{Olsh} A.Yu. Ol'shanskii, On residualing homomorphisms and $G$-subgroups of hyperbolic groups.
\emph{Internat. J. Algebra Comput.} \textbf{3} (1993), no. 4, 365--409.

\bibitem{Osin-acyl} D. Osin, Acylindrically hyperbolic groups.
\emph{Trans. Amer. Math. Soc.} \textbf{368} (2016), no. 2, 851--888.

\bibitem{Rem-polyc} V.N. Remeslennikov, Conjugacy in polycyclic groups. (Russian) \emph{Algebra i Logika} \textbf{8} (1969), no. 5, 712--725.
English translation in \emph{Algebra and Logic} \textbf{8} (1969), no. 6, 404--411.

\bibitem{Rem} V.N. Remeslennikov, Finite approximability of groups with respect to conjugacy. (Russian)
\emph{Sibirsk. Mat. \v{Z}.} \textbf{12} (1971), 1085--1099. English translation in \emph{Siberian Math. J.} \textbf{23} (1971), 783--792.

\bibitem{RZ} L. Ribes, P.A. Zalesskii,  { Profinite groups.} Second edition. Ergebnisse der Mathematik und ihrer Grenzgebiete. 3. Folge.
A Series of Modern Surveys in Mathematics, 40. \emph{Springer-Verlag, Berlin}, 2010. xvi+464~pp.

\bibitem{Rips}  E. Rips, Subgroups of small cancellation groups. \emph{Bull. London Math. Soc.} \textbf{14} (1982), no. 1, 45--47.

\bibitem{Segal-book} D. Segal, Polycyclic groups. Cambridge Tracts in Mathematics, 82. \emph{Cambridge University Press,
Cambridge}, 1983. xiv+289 pp.

\bibitem{Stebe-SL3} P.F. Stebe, Conjugacy separability of groups of integer matrices.
\emph{Proc. Amer. Math. Soc.} \textbf{32} (1972), 1--7.

\bibitem{Toinet}  E. Toinet, Conjugacy $p$-separability of right-angled Artin groups and applications. \emph{Groups Geom. Dyn.} \textbf{7} (2013), no. 3, 751--790.

\bibitem{Truf} B. Truffault, Centralisateurs des \'el\'ements dans les groupes de Greendlinger. (French)
\emph{C. R. Acad. Sci. Paris. S\'er. A} \textbf{279} (1974), 317--319.

\bibitem{Wehr}  B.A.F. Wehrfritz, Two examples of soluble groups that are not conjugacy separable. \emph{J. London Math. Soc. (2)} \textbf{7} (1973), 312--316.

\bibitem{Wise-small_canc-cubical}  D.T. Wise, Cubulating small cancellation groups. \emph{Geom. Funct. Anal.} \textbf{14} (2004), no. 1, 150--214.

\bibitem{Wise-Rips} D.T. Wise, A residually finite version of Rips's construction.
\emph{Bull. London Math. Soc.} \textbf{35} (2003), no. 1, 23--29.


\bibitem{Wise-QJM} D.T. Wise, Subgroup separability of graphs of free groups with cyclic edge groups.
\emph{Q. J. Math.} \textbf{51} (2000), no. 1, 107--129.

\end{thebibliography}
\end{document}